\documentclass[a4paper, reqno, 11pt]{amsart}

\usepackage{amscd, amssymb, amsthm, amsmath, color}
\usepackage{graphicx, psfrag, enumerate}

\usepackage{color}

\usepackage{hyperref}

\newtheorem{lemma}{Lemma}[section]
\newtheorem{theorem}[lemma]{Theorem}

\newtheorem{corollary}[lemma]{Corollary}

\newtheorem{thm}{Main Theorem}

\newcommand{\stat}{{\mathsf{stat}}}
\newcommand{\inv}{{\mathsf{inv}}}
\newcommand{\maj}{{\mathsf{maj}}}
\newcommand{\des}{{\mathsf{des}}}
\newcommand{\Des}{{\mathsf{Des}}}
\newcommand{\fdes}{{\mathsf{fdes}}}

\newcommand{\fmaj}{{\mathsf{fmaj}}}

\newcommand{\dmaj}{{\mathsf{dmaj}}}

\newcommand{\nega}{{\mathsf{neg}}}
\newcommand{\Nega}{{\mathsf{Neg}}}
\newcommand{\ddes}{{\mathsf{ddes}}}
\newcommand{\col}{{\mathsf{col}}}

\newcommand{\sgm}{{\mathsf{sgm}}}

\def\B{\mathcal{B}}
\def\D{\mathcal{D}}
\def\G{\mathsf{G}}
\def\F{\mathcal{F}}
\def\P{\mathsf{P}}
\def\L{\mathsf{L}}
\def\S{\mathbf{S}}
\def\H{\mathbf{H}}
\def\z{\mathbf{z}}

\newcommand{\ch}{{\textsf{ch}}}
\newcommand{\bars}{{\textsf{bar}}}
\newcommand{\wt}{{\textsf{wt}}}

\renewcommand{\d}{\displaystyle}

\begin{document}
\title{Signed Euler-Mahonian identities}

\author{Sen-Peng Eu}
\address{Department of Mathematics \\
National Taiwan Normal University \\
Taipei, Taiwan 11677, ROC} \email[Sen-Peng Eu]{senpengeu@gmail.com}

\author{Zhicong Lin}
\address{Research Center for Mathematics and Interdisciplinary Sciences \\
Shandong University \\
Qingdao 266237, P.R.~China}
\email[Zhicong Lin]{linz@sdu.edu.cn}

\author{Yuan-Hsun Lo}
\address{Department of Applied Mathematics \\
National Pingtung University \\
Pingtung, Taiwan 90003, ROC}
\email[Yuan-Hsun Lo]{yhlo@mail.nptu.edu.tw}



\thanks{Partially supported by Ministry of Science and Technology, Taiwan under Grants 107-2115-M-003-009-MY3 (S.-P.~Eu) and 108-2115-M-153-004-MY2 (Y.-H.~Lo), the National Natural Science Foundation of China under Grant 11871247 (Z.~Lin), and the project of Qilu Young Scholars of Shandong University (Z.~Lin).}

\dedicatory{Dedicated to Xuding Zhu on the occasion of his 60th birthday}

\date{\today}

\maketitle


\begin{abstract}
A relationship between signed Eulerian polynomials and the classical Eulerian polynomials on $\mathfrak{S}_n$ was given by D\'{e}sarm\'{e}nien and Foata in 1992, and a refined version, called signed Euler-Mahonian identity, together with a bijective proof were proposed by Wachs in the same year.
By generalizing this bijection, in this paper we extend the above results to the Coxeter groups of types $B_n$, $D_n$, and the complex reflection group $G(r,1,n)$, where the `sign' is taken to be any one-dimensional character.
Some obtained identities can be further restricted on some particular set of permutations.
We also derive some new interesting sign-balance polynomials for types $B_n$ and $D_n$.
\end{abstract}

\section{Introduction}


Let $\mathfrak{S}_n$ be the symmetric group of $[n]:=\{1,2,\ldots,n\}$.
The \emph{inversion number} and \emph{descent set} of $\pi=\pi_1\pi_2\cdots\pi_n\in\mathfrak{S}_n$ are defined respectively by
\begin{align}
\inv(\pi) &:= |\{(i,j)\in[n]^2:\, i<j \text{ and } \pi_i>\pi_j\}|, \label{eq:inv_def} \\
\Des(\pi) &:= \{i\in[n-1]:\, \pi_i > \pi_{i+1}\}, \notag
\end{align}
and denoted by $\des(\pi)$ the cardinality of $\Des(\pi)$.
Define the \emph{major index} of $\pi$ by
\begin{align*}
\maj(\pi):= \sum_{i\in\Des(\pi)} i.
\end{align*}

In the theory of Coxeter groups, the length $\ell(\pi)$ of a group element $\pi$ is the minimal number of Coxeter generators needed to express $\pi$.
It is well-known that $\mathfrak{S}_n$ is the Coxeter group of type $A_{n-1}$, where the generators are the adjacent transpositions, say $s_i:=(i,i+1)$ for $i\in[n]$, and $\ell(\pi)=\inv(\pi)$ for $\pi\in\mathfrak{S}_n$.
A fundamental result of MacMahon~\cite{MacMahon_13} states that $\maj$ and $\inv$ have the same distribution over $\mathfrak{S}_n$:
$$
\sum_{\pi\in\mathfrak{S}_n}q^{\maj(\pi)}=[n]_q!=\sum_{\pi\in\mathfrak{S}_n}q^{\inv(\pi)},
$$
where $[n]_q!=[n]_q[n-1]_q\cdots[1]_q$ with $[i]=1+q+\cdots+q^{i-1}$.
In a Coxeter group, a statistic is called \emph{Mahonian} if it is equidistributed with the length function $\ell$ of the group.
As the generating function of $\sum_{\pi\in\mathfrak{S}_n}t^{\des(\pi)}$ can be traced back to Euler's wrok, see e.g. ~\cite{Petersen_15,Stanley_97}, any statistic equidistributed with $\des$ over $\mathfrak{S}_n$ is called \emph{Eulerian}.
The joint distribution of one Eulerian statistic with one Mahonian statistic is known as the {\em Euler-Mahonian distribution}, which  was introduced by Foata and Zeilberger~\cite{FZ} in 1990 and extensively studied since then. The most important example is the bivariate distribution of the descent number and major index over the symmetric group. 


Motivated by the work of Loday~\cite{Lod}, D\'{e}sarm\'{e}nien and Foata~\cite{DF_92} first investigated the relationship between a \emph{signed Eulerian} polynomial and the classical Eulerian polynomial, called \emph{signed Eulerian identity}, and obtained, for any positive integer $n$, that
\begin{equation}\label{eq:signedE-A-even}
\sum_{\pi\in\mathfrak{S}_{2n}} (-1)^{\ell(\pi)}t^{\des(\pi)} = (1-t)^n\sum_{\pi\in\mathfrak{S}_n}t^{\des(\pi)}
\end{equation}
and
\begin{equation}\label{eq:signedE-A-odd}
\sum_{\pi\in\mathfrak{S}_{2n+1}} (-1)^{\ell(\pi)}t^{\des(\pi)} = (1-t)^n\sum_{\pi\in\mathfrak{S}_{n+1}}t^{\des(\pi)}.
\end{equation}
By proposing an elegant involution proof for \eqref{eq:signedE-A-even} and~\eqref{eq:signedE-A-odd}, Wachs~\cite{Wachs_92} derived a $q$-analogue of \eqref{eq:signedE-A-even} as
\begin{equation}\label{eq:signedEM-A}
\sum_{\pi\in\mathfrak{S}_{2n}} (-1)^{\ell(\pi)}t^{\des(\pi)}q^{\maj(\pi)} = \prod_{i=1}^{n}(1-tq^{2i-1})^n\sum_{\pi\in\mathfrak{S}_n}t^{\des(\pi)}q^{2\maj(\pi)},
\end{equation}
which is called a \emph{signed Euler-Mahonian identity}. Setting $t=1$, identity~\eqref{eq:signedEM-A} reduces to the even case of the following {\em signed Mahonian identity} due to Gessel and Simion~\cite{Wachs_92}:
\begin{equation}\label{iden:gs}
\sum_{\pi\in\mathfrak{S}_{n}} (-1)^{\ell(\pi)}q^{\maj(\pi)}=[1]_q[2]_{-q}[3]_q[4]_{-q}\cdots[n]_{(-1)^{n-1}q}.
\end{equation}

The above signed enumeration identities have been generalized and extended in two main directions:
\begin{itemize}
\item From symmetric group to other reflection groups by Reiner~\cite{Reiner_95}, Adin--Gessel--Roichman~\cite{Adin_Gessel_Roichman_05}, Biagioli~\cite{Biagioli_06}, Biagioli--Caselli~\cite{BC_12} and Chang--Eu--Fu--Lin--Lo~\cite{CHFLL}, to name a few.
\item To restricted permutations such as $321$-avoiding permutations by Adin--Roichman~\cite{AR_04} and Eu--Fu--Pan--Ting~\cite{EFPT_15}, Simsun permutations by Eu--Fu--Pan~\cite{EFP}, PRW permutations by Lin--Wang--Zeng~\cite{LWZ} and permutations with subsequence restrictions by Eu--Fu--Hsu--Liao--Sun~\cite{EFHLS}.
\end{itemize}
In this paper, we continue with the first direction and  find some nice signed Eulerian or Euler-Mahonian identities for Coxeter groups of type $B_n$, type $D_n$ and the complex reflection group $G(r,1,n)$ that have been long-overlooked.

In the rest of the introduction, we recall some Eulerian and Mahonian statistics for Coxeter group of type $B_n$ and highlight  three of our main results.  

\subsection{Coxeter group of types $B_n$}

Let $\B_n$ be the group of \emph{signed permutations} of $[n]$, which consists of all bijections $\pi$ of $\{\pm 1, \pm 2, \ldots,\pm n\}$ onto itself such that $\pi(-i)=-\pi(i)$.
Elements in $\B_n$ are centrally symmetric and hence can be simply written in the form $\pi=\pi_1\pi_2\cdots\pi_n$, where $\pi_i:=\pi(i)$.
For convenience, we use $\bar{i}$ to represent $-i$.
$\B_n$ is also the Coxeter group of type $B_n$ with generators $s_0=(1,-1)$ and $s_i=(i,i+1)$ for $i\in[n-1]$.
Let $\ell_B$ be the length function of $\B_n$ with respect to this set of generators.
Following~\cite{BB_05}, for $\pi\in\B_n$ let $\Nega(\pi):=\{i:\,\pi_i<0\}$, then $\ell_B(\pi)$ can be represented as
\begin{align*}
\ell_B(\pi)=\inv(\pi)-\sum_{i\in\Nega(\pi)}\pi_i,
\end{align*}
where $\inv$ is defined as \eqref{eq:inv_def} with respect the natural linear order, i.e.,
$$\bar{n}<\cdots<\bar{1}<1<\cdots<n.$$
Denote by $\nega(\pi)$ the cardinality of $\Nega(\pi)$.

Adin and Roichman~\cite{Adin_Roichman_01} defined the \emph{flag descent number} and the \emph{flag major index} for $\pi\in\B_n$ by
\begin{align*}
\fdes(\pi) &:= 2\cdot\des_F(\pi)+\delta(\pi_1<0), \\
\fmaj(\pi) &:= 2\cdot\maj_F(\pi)+\nega(\pi),
\end{align*}
where $\des_F$ and $\maj_F$ are defined like $\des$ and $\maj$ respectively, but with respect to the linear order $$\bar{1}<\bar{2}<\cdots<\bar{n}<1<2<\cdots<n,$$
and $\delta(\mathsf{A})=1$ if the statement $\mathsf{A}$ is true and $0$ otherwise.

By generalizing Wach's involution, Biagioli~\cite{Biagioli_06} proposed the following signed Mahonian identity:
\begin{equation}\label{eq:signedM-B}
\sum_{\pi\in\B_{2n}}(-1)^{\ell_B(\pi)}q^{\fmaj(\pi)} = \prod_{i=1}^{n}\left(1-q^{4i-2}\right)\sum_{\pi\in\B_n}q^{2\fmaj(\pi)},
\end{equation}
an analogous version of the even case of~\eqref{iden:gs} for Coxeter group of type $B_n$.

\subsection{Extensions to one-dimensional characters}
In this paper we aim to extend the signed Eulerian-Mahonian identities to Coxeter groups of types $B_n$, $D_n$ and complex reflection group $G(r,1,n)$, where each one of the one-dimensional characters of the group is taken to be the ``sign''.
Precisely speaking, consider the polynomial
\begin{equation}\label{eq:signed-single}
\sum_{\pi\in\mathcal{W}}\chi(\pi)t^{\stat_1(\pi)}q^{\stat_2(\pi)},
\end{equation}
where $\mathcal{W}=\B_{2n},\D_{2n}$ or $G(r,1,2n)$, $\chi$ is a one-dimensional character of $\mathcal{W}$, and $\stat_1$ and $\stat_2$ are chosen as Eulerian and Mahonian statistics, respectively.
For the sake of convenience, we use $G_{r,n}$ to denote $G(r,1,n)$ throughout this paper.
The first main result in this paper is the following.
\begin{thm}[Theorem~\ref{thm:main-even}]
Let $r$ and $n$ be two positive integers.
For any $b$, $0\leq b\leq r-1$, we have
\begin{align*}
\sum_{\pi\in G_{r,2n}}  \chi_{1,b}(\pi)t^{\fdes(\pi)}q^{\fmaj(\pi)}x^{\col(\pi)} = \prod_{i=1}^{n}\left(1-t^rq^{r(2i-1)}\right)\sum_{\pi\in G_{r,n}}t^{\fdes(\pi)}(\omega^bq)^{2\fmaj(\pi)}x^{2\col(\pi)}, 
\end{align*}
where $\omega$ is a primitive $r$th root of $1$, and $\col(\pi)$ is the sum of colors in $\pi$.
\end{thm}

As $\B_n=G_{2,n}$, the corresponding results for Coxeter group of type $B_n$ can be derived by plugging $r=2$ into above as
\begin{align*}
\sum_{\pi\in\B_{2n}}(-1)^{\ell_B(\pi)}t^{\fdes(\pi)}q^{\fmaj(\pi)}x^{\nega(\pi)} = \prod_{i=1}^{n}(1-t^2q^{4i-2}) \sum_{\pi\in\B_{n}}t^{\fdes(\pi)}q^{2\fmaj(\pi)}x^{2\nega(\pi)},
\end{align*}
and
\begin{align*}
\sum_{\pi\in\B_{2n}}(-1)^{\inv(|\pi|)}t^{\fdes(\pi)}q^{\fmaj(\pi)}x^{\nega(\pi)} = \prod_{i=1}^{n}(1-t^2q^{4i-2}) \sum_{\pi\in\B_{n}}t^{\fdes(\pi)}q^{2\fmaj(\pi)}x^{2\nega(\pi)},
\end{align*}
which generalize \eqref{eq:signedEM-A} and \eqref{eq:signedM-B}.
Here, $|\pi|=|\pi_1||\pi_2|\cdots|\pi_n|$.
Analogous result for Coxeter group of type $D_n$ is obtained in Theorem~\ref{thm:signEM-D}.

Note that we actually derive a more refined version of the above, stated in Theorem~\ref{thm:main-even}, which is a refinement to \emph{colored restricted subsets}.
The case for $\chi_{0,b}$, $0\leq b\leq r-1$, is addressed in Theorem~\ref{thm:main-odd} as well.

For a natural extension of \eqref{eq:signedE-A-odd}, we also consider the signed Eulerian polynomials for $\B_{2n+1}$ and $\D_{2n+1}$, and derive some neat identities.
We highlight the following results.

\begin{thm}[Theorem~\ref{thm:B-odd-absinv} and Theorem~\ref{thm:B-odd-length}]
Let $n$ be a positive integer. We have
\begin{align*}
\sum_{\pi\in\B_{2n+1}}(-1)^{\inv(|\pi|)}t^{\fdes(\pi)} = (1-t^2)^n \sum_{\pi\in\B_{n+1}}t^{\fdes(\pi)},
\end{align*}
and
\begin{align*}
\sum_{\pi\in\B_{2n+1}}(-1)^{\ell_B(\pi)}t^{\fdes(\pi)} = (1-t^2)^n(1-t)\sum_{\pi\in\B_n}t^{2\des_B(\pi)}.
\end{align*}
\end{thm}

\begin{thm}[Theorem~\ref{thm:D-odd-length}]
For $\pi\in\D_n$, let $\ddes(\pi) := \fdes(\pi_1\pi_2\cdots\pi_{n-1}|\pi_n|)$.
Then, for any positive integer $n$, we have
\begin{align*}
\sum_{\pi\in\D_{2n+1}}(-1)^{\ell_D(\pi)}t^{\ddes(\pi)} = (1-t^2)^n\sum_{\pi\in\D_{n+1}}t^{\ddes(\pi)}.
\end{align*}
\end{thm}

Note again that a refined version of the first identity of Main Theorem 2 is given in Theorem~\ref{thm:B-odd-absinv}.

Another aspect in this paper is to consider the signed polynomial \eqref{eq:signed-single} for $\mathcal{W}=\B_n$ and $\D_n$, and $\stat_1$ and $\stat_2$ are chosen as type $B/D$ descent number and major index, defined by Bj\"{o}rner and Brenti~\cite{BB_05}, respectively.
We do actually obtain some nice closed forms; see Theorem~\ref{thm:B-GF-length} and Theorem~\ref{thm:B-GF-absinv} for $\B_n$, and Theorem~\ref{thm:D-GF} for $\D_n$.

\medskip

The rest of this paper is organized as follows.
The proof of Main Theorem 1 is given in Section~\ref{sec:Grn}.
Some new signed Eulerian and signed Euler-Mahonian identities for Coxeter groups of types $B_n$ and $D_n$ are collected in Section~\ref{sec:Bn} and Section~\ref{sec:Dn}, respectively.
In Section~\ref{sec:sign-balance} we derive the generating functions of the signed polynomials for $\B_n$ and $\D_n$.
Finally, a brief conclusion is in Section~\ref{sec:conclusion}.


%
%
%
%

\section{Extensions to complex reflection group $G_{r,n}$}\label{sec:Grn}

\subsection{Colored permutations}\label{sec:Grn-def}

Let $r,n$ be positive integers.
The group of \emph{colored permutations} $G_{r,n}$ of $n$ letters with $r$ colors is the \emph{wreath product} $G_{r,n}:=\mathbb{Z}_r\wr\mathfrak{S}_n$ of the cyclic group $\mathbb{Z}_r(:=\mathbb{Z}/r\mathbb{Z})$ with $\mathfrak{S}_n$.
$G_{r,n}$ consists of ordered pairs $(\sigma,\z)$, where $\sigma=\sigma_1\sigma_2\cdots\sigma_n\in\mathfrak{S}_n$ and $\z=(z_1,z_2,\ldots,z_n)$ is an $n$-tuple of integers (or, \emph{colors}) with $z_i\in\mathbb{Z}_r$.
An element $(\sigma,\z)\in G_{r,n}$ can be represented as a word $$\sigma_1^{[z_1]}\sigma_2^{[z_2]}\cdots\sigma_n^{[z_n]},$$ in which the superscript $[t]$ is omitted when $t=0$.
It is called the \emph{window notation} of  $(\sigma,\z)$.
For example, $(513426,(0,1,0,2,1,3))\in G_{4,6}$ can be represented as $51^{[1]}34^{[2]}2^{[1]}6^{[3]}$.
It is clear that $G_{1,n}=\mathfrak{S}_n$ and $G_{2,n}=\B_n$ by viewing the letters with color label $1$ as negative numbers.

The group $G_{r,n}$ can be generated by the set of generators $$\mathcal{S}_n=\{s_0,s_1,\ldots,s_{n-1}\},$$
where $s_i=(i,i+1)$ for $1\leq i\leq n-1$, while $s_0$ is the action of adding the color label on the first letter by one (note that the color label is taken module $r$).
The generators are subject to the defining relations
$$
\begin{cases}
s_0^r=1, \\
s_i^2=1 \text{ for } 1\leq i\leq n-1, \\
(s_is_j)^2=1 \text{ for } |i-j|>1, \\
(s_is_{i+1})^3=1 \text{ for } 1\leq i\leq n-2, \\
(s_0s_1)^{2r} = 1.
\end{cases}
$$
Note that $12\cdots n$ is the identity.

The \emph{length} $\ell(\pi)$ of $\pi=(\sigma,\z)\in G_{r,n}$ is defined to be the minimal number of generators in $G_{r,n}$ needed to express it.
One of the combinatorial interpretations of $\ell(\pi)$ was proposed by Bagno~\cite{Bagno_04} by
\begin{equation}\label{eq:length}
\ell(\pi)=\inv(\pi)+\sum_{z_i>0}(\sigma_i+z_i-1),
\end{equation}
where $\inv(\pi)$ is the number of inversions of $\pi$ with respect to the linear order
\begin{equation}\label{eq:order_L}
n^{[r-1]}<\cdots<n^{[1]}<\cdots<1^{[r-1]}<\cdots<1^{[1]}<1<\cdots<n.
\end{equation}
For $\pi=(\sigma,\z)\in\G_{r,n}$ let $\col(\pi):=\sum_{i=1}^{n}z_i$.
For example, let $\pi=51^{[1]}34^{[2]}2^{[1]}6^{[3]}\in G_{4,6}$.
Then $\ell(\pi)=13+(1+5+2+8)=29$ and $\col(\pi)=7$.

\subsection{Restricted colored permutations}\label{sec:Grn-restricted}

Denote by $\S=(S_1,S_2,\ldots,S_n)$ the $n$-tuple of non-empty subsets with $S_i\subseteq\mathbb{Z}_r$ and by $\mathfrak{R}_{r,n}$ the collection of all such $n$-tuples.
The set of \emph{$\S$-restricted colored permutations}, denoted by $G_{r,n}(\S)$, is defined by
\begin{align*}
G_{r,n}(\S) := \{\pi=(\sigma,\z):\,z_i\in S_i \text{ for }i\in[n]\}.
\end{align*}
For example, $$G_{4,2}(\{0,2\},\{2,3\}) = \{12^{[2]},12^{[3]},1^{[2]}2^{[2]},1^{[2]}2^{[3]},21^{[2]},21^{[3]},2^{[2]}1^{[3]},2^{[2]}1^{[3]}\}.$$
Note that if $\emptyset$ appears as an element in $\S$, then $G_{r,n}(\S)=\emptyset$.

Let $\S\in\mathfrak{R}_{r,n}$ and $\H\in\mathfrak{R}_{r,2n}$.
We say $\S\prec\H$ if
\begin{equation}\label{eq:S_iH_i}
S_i = H_{2i-1}\cap H_{2i}, \quad \text{for }i\in[n].
\end{equation}
For example, if $\S=(\{0,2\},\{2,3\})\in\mathfrak{R}_{4,2}$ and $\H=(\{0,1,2,3\},\{0,2\},\{2,3\},\{0,2,3\})$, $\H'=(\{0,1,2,3\},\{0,2\},\{2,3\},\{0,1,3\})\in\mathfrak{R}_{4,4}$, then $\S\prec\H$ but $\S\nprec\H'$.

By \eqref{eq:S_iH_i} we have that, for $\H\in\mathfrak{R}_{r,2n}$, there is unique $\S\in\mathfrak{R}_{r,n}$ such that $\S\prec\H$.


\subsection{Signed Eulerian-Mahonian identities}\label{sec:Grn-main}

The $2r$ one-dimensional characters of $G_{r,n}$ are characterized in terms of the length function as follows.

\begin{theorem}[\cite{CHFLL}]\label{thm:1-dm}
For any positive integers $r$ and $n$, $G_{r,n}$ has $2r$ one-dimensional characters
\begin{equation}\label{eq:1-dim}
\chi_{a,b}(\pi)=(-1)^{a\left(\ell(\pi)-\col(\pi)\right)}\omega^{b\,\col(\pi)},
\end{equation}
where $\omega$ is a primitive $r$th root of $1$, $a=0,1$ and $b=0,1,\ldots,r-1$.
\end{theorem}

It should be emphasized that the one-dimensional characters of $G_{r,n}$ are well known, studied first by Schur and Specht, see e.g.~\cite{Kerber_71, Macdonald_95}, as the form:
\begin{equation}\label{eq:1-dim-old}
\chi_{a,b}(\pi)=(-1)^{a\,\inv(|\pi|)}\omega^{b\,\col(\pi)},
\end{equation}
where $|\pi|=\sigma$.
The two expressions in \eqref{eq:1-dim} and \eqref{eq:1-dim-old} are similar, while the former one is more helpful to derive our first main result.
See~\cite{CHFLL} for a simple comparison of these two formulae.
Note that a generalization of the classical one~\eqref{eq:1-dim-old} can be found in the study of signed Mahonian polynomials over projective reflection groups~\cite{BC_12}.

For $\pi\in G_{r,n}$ let $$\Des_F(\pi):=\{i\in[n-1]:\,\pi_i>\pi_{i+1}\}$$
with respect to the linear order
\begin{equation}\label{eq:flag-order}
1^{[r-1]}<\cdots<n^{[r-1]}<\cdots<1^{[1]}<\cdots<n^{[1]}<1<\cdots<n.
\end{equation}
Adin and Roichman~\cite{Adin_Roichman_01} defined the \emph{flag descent number} $\fdes$ and the \emph{flag major index} $\fmaj$ by
\begin{align*}
\fdes(\pi) &:= r\cdot\des_F(\pi)+z_1, \\
\fmaj(\pi) &:= r\cdot\maj_F(\pi)+\col(\pi),
\end{align*}
where $\des_F(\pi)=|\Des_F(\pi)|$ and $$\maj_F(\pi):=\sum_{i\in\Des_F(\pi)}i.$$
For example, let $\pi=4^{[2]}51^{[1]}32^{[1]}6^{[3]}\in G_{4,6}$.
Then $\Des_F(\pi)=\{2,4,5\}$, $\fdes(\pi)=4\cdot 3+2=14$ and $\fmaj=4\cdot 11+7=51$.

For a given $\H\in\mathfrak{R}_{r,2n}$, define a map $\varphi:G_{r,2n}(\H)\to G_{r,2n}(\H)$ as follows.
Let $i\in[n]$ be the smallest integer such that $2i-1$ and $2i$ have different colors or are not in adjacent positions in the window notation of $\pi$.
Then, let $\varphi(\pi)$ be the a colored permutation obtained from $\pi$ by swapping the two letters $2i-1$ and $2i$.
If no such $i$ exists let $\varphi(\pi)=\pi$.
For example, in $G_{4,6}(\H)$ for some $\H$, $\varphi(51^{[1]}34^{[2]}2^{[1]}6^{[3]}) = 52^{[1]}34^{[2]}1^{[1]}6^{[3]}$ since $1$ and $2$ are not in adjacent positions, $\varphi(52^{[1]}1^{[1]}6^{[3]}43^{[2]}) = 52^{[1]}1^{[1]}6^{[3]}34^{[2]}$ since $3$ and $4$ have different colors even though they are in adjacent positions, and $562^{[3]}1^{[3]}4^{[1]}3^{[1]}$ is fixed under $\varphi$.

It is obvious that the map $\varphi$ is well-defined and an involution on $G_{r,2n}(\H)$ whose fixed points (those $\pi$ such that $\varphi(\pi)=\pi$) are the colored permutations in which the letters $2i-1$ and $2i$ are adjacent and having the same color.
Moreover, assume the positions of the letters $2i-1$ and $2i$ are $2p-1$ and $2p$.
Then, the set of common colors is $H_{2p-1}\cap H_{2p}$.
Denote by $\mathcal{F}_{r,2n}(\H)$ the set of fixed points under $\varphi$.

\begin{lemma}\label{lem:fix-color-even}
Let $r$ and $n$ be two positive integers.
For integer $0\leq b\leq r-1$ and $\H\in\mathfrak{R}_{r,2n}$, we have
\begin{align*}
\sum_{\pi\in G_{r,2n}(\H)} \chi_{1,b}(\pi)t^{\fdes(\pi)}q^{\fmaj(\pi)}\prod_{i=1}^{2n}x_i^{z_i} = \sum_{\pi\in\mathcal{F}_{r,2n}(\H)} \chi_{1,b}(\pi)t^{\fdes(\pi)}q^{\fmaj(\pi)}\prod_{i=1}^{2n}x_i^{z_i}.
\end{align*}
\end{lemma}
\proof
Consider $\pi=(\sigma,\z)\in G_{r,2n}(\H)\backslash\mathcal{F}_{r,2n}(\H)$ which is not a fixed point under $\varphi$.
Let $\varphi(\pi)=\pi'=(\sigma',\z')$.
It is clear that $\z=\z'$, so
\begin{equation}\label{eq:phi-color}
\prod_{i=1}^{2n}x_i^{z_i} = \prod_{i=1}^{2n}x_i^{z_i'} \quad \text{and} \quad \col(\pi)=\col(\pi').
\end{equation}
It is also clear that $\ell(\pi)=\ell(\pi')\pm 1$ and thus
\begin{equation}\label{eq:phi-reverse}
\chi_{1,b}(\pi)=(-1)\chi_{1,b}(\pi')
\end{equation}
by the description of $\chi_{1,b}$ shown in \eqref{eq:1-dim}.
Furthermore, $\varphi$ preserves the set $\Des_F$ and hence does not change both $\fdes$ and $\fmaj$.
That is,
\begin{equation}\label{eq:phi-preserve}
\fdes(\pi)=\fdes(\pi') \quad \text{and} \quad \fmaj(\pi)=\fmaj(\pi').
\end{equation}
By \eqref{eq:phi-color}--\eqref{eq:phi-preserve} we have
\begin{align*}
\chi_{1,b}(\pi)t^{\fdes(\pi)}q^{\fmaj(\pi)}\prod_{i=1}^{2n}x_i^{z_i} = -\chi_{1,b}(\pi')t^{\fdes(\pi')}q^{\fmaj(\pi')}\prod_{i=1}^{2n}x_i^{z'_i},
\end{align*}
and thus the result follows.
\qed

Now we are ready to prove our first main result.

\begin{theorem}\label{thm:main-even}
Let $r$ and $n$ be two positive integers.
For integer $0\leq b\leq r-1$ and $\H\in\mathfrak{R}_{r,2n}$, we have
\begin{align}
\sum_{\pi\in G_{r,2n}(\H)} & \chi_{1,b}(\pi)t^{\fdes(\pi)}q^{\fmaj(\pi)}\prod_{i=1}^{2n}x_i^{z_i} \notag \\
& = \prod_{i=1}^{n}\left(1-t^rq^{r(2i-1)}\right)\sum_{\pi\in G_{r,n}(\S)}t^{\fdes(\pi)}(\omega^bq)^{2\fmaj(\pi)}\prod_{i=1}^{n}(x_{2i-1}x_{2i})^{z_i}, \label{eq:main-even-set}
\end{align}
where $\S\in\mathfrak{R}_{r,n}$ is the unique $n$-tuple such that $\S\prec\H$.
In particular,
\begin{align}
\sum_{\pi\in G_{r,2n}} & \chi_{1,b}(\pi)t^{\fdes(\pi)}q^{\fmaj(\pi)}x^{\col(\pi)} \notag \\
& = \prod_{i=1}^{n}\left(1-t^rq^{r(2i-1)}\right)\sum_{\pi\in G_{r,n}}t^{\fdes(\pi)}(\omega^bq)^{2\fmaj(\pi)}x^{2\col(\pi)}. \label{eq:main-even}
\end{align}
\end{theorem}
\proof
\eqref{eq:main-even-set} can be reduced to \eqref{eq:main-even} by simply letting $\H=(\mathbb{Z}_r,\ldots,\mathbb{Z}_r)$, $\S=(\mathbb{Z}_r,\ldots,\mathbb{Z}_r)$ and $x=x_1=x_2=\cdots=x_{2n}$.
By Lemma~\ref{lem:fix-color-even}, to derive \eqref{eq:main-even-set} it suffices to consider the set of fixed points $\mathcal{F}_{r,2n}(\H)$.

Consider a bijective correspondence $\phi$ between elements in $\mathcal{F}_{r,2n}(\H)$ and colored permutations in $G_{r,n}(\S)$ with some letters marked a ``hat'', according to the following rules.
\begin{itemize}
\item Each pair of adjacent entries of type $(2j-1)^{[t]},2j^{[t]}$ in $\mathcal{F}_{r,2n}(\H)$ is replaced by $j^{[t]}$.
\item Each pair of adjacent entries of type $2j^{[t]},(2j-1)^{[t]}$ in $\mathcal{F}_{r,2n}(\H)$ is replaced by $\hat{j}^{[t]}$.
\end{itemize}
We denote by $\widehat{G}_{r,n}(\S)$ the set of all \emph{hatted} $\S$-restricted colored permutations of $[n]$.
Formally, $\phi:\pi=(\sigma,\z)\in\F_{r,2n}(\H)\mapsto\pi'=(\sigma',\z')\in\widehat{G}_{r,n}(\S)$ by
\begin{equation}\label{eq:phi-def}
\sigma'_i = \begin{cases}
j, & \text{if } \sigma_{2i-1} = 2j, \\
\hat{j}, & \text{if } \sigma_{2i} = 2j,
\end{cases}
\quad \text{and} \quad
z'_i = z_{2i-1} = z_{2i}
\end{equation}
for $i=\in[n]$.
For $\pi'=(\sigma',\z')\in\widehat{G}_{r,n}(\S)$ let $\ell(\pi'),\chi_{a,b}(\pi'), \Des_F(\pi'), \fdes(\pi'), \fmaj(\pi')$ and $\col(\pi')$ be defined by omitting the hats and let $\P(\pi')$ be the set of the positions of the hatted letters.
For example, let $5^{[1]}6^{[1]}2^{[2]}1^{[2]}874^{[1]}3^{[1]}\in\F_{3,8}(\H)$ for some $\H\in\mathfrak{R}_{3,8}$.
Then $\pi'=\phi(\pi)=3^{[1]}\hat{1}^{[2]}\hat{4}\hat{2}^{[1]}\in\widehat{G}_{3,4}(\S)$, where $\S\prec\H$, is the corresponding hatted colored permutation with $\ell(\pi')=9$, $\chi_{a,b}(\pi')=(-1)^{5a}\omega^{4b}$, $\Des_F(\pi')=\{1,3\}$, $\fdes(\pi')=7$, $\fmaj(\pi')=16$, $\col(\pi')=4$ and $\P(\pi')=\{2,3,4\}$.

Let $\pi'=(\sigma',\z')=\phi(\pi=(\sigma,\z))\in\widehat{G}_{r,n}(\S)$.
By the definition of $\phi$, i.e., \eqref{eq:phi-def}, we have
\begin{equation}\label{eq:col-hat}
\prod_{i=1}^{2n}x_i^{z_i} = \prod_{i=1}^{n}(x_{2i-1}x_{2i})^{z_i'} \quad \text{and} \quad \col(\pi)=2\col(\pi').
\end{equation}
By computing the flag descent set with respect to the linear order \eqref{eq:flag-order}, it follows that
\begin{equation}\label{eq:DesF-hat}
\Des_F(\pi)=\{2i:\,i\in\Des_F(\pi')\}\cup\{2i-1:\,i\in\P(\pi')\}
\end{equation}
and thus
\begin{equation}\label{eq:fdes-hat}
\fdes(\pi) = r(\des_F(\pi') + |\P(\pi')|) + z'_1 = \fdes(\pi')+r|\P(\pi')|.
\end{equation}
By~\eqref{eq:DesF-hat}, we have $\maj_F(\pi)=2\maj_F(\pi')+\sum_{i\in\P(\pi')}(2i-1)$. 
It follows from~\eqref{eq:col-hat} that
\begin{align}
\fmaj(\pi) &= r\Big(2\maj_F(\pi')+\sum_{i\in\P(\pi')}(2i-1)\Big) + 2\col(\pi') \notag \\
&= 2\fmaj(\pi') + r\sum_{i\in\P(\pi')}(2i-1). \label{eq:fmaj-hat}
\end{align}
Furthermore, decompose $\P(\pi')$ into two parts $\P^N(\pi')\uplus\P^C(\pi')$, where $\P^N(\pi')$ collects the positions of the hatted letters without colors (i.e., $z'_i=0$).
Following the linear order~\eqref{eq:order_L} we have
\begin{equation}\label{eq:inv-hat}
\inv(\pi) = 4\cdot\inv(\pi') + |\P^N(\pi')| - |\P^C(\pi')| + \sum_{z'_i>0}1.
\end{equation}
Then,
\begin{align*}
\ell(\pi) &= \inv(\pi) + \sum_{z_i>0}(\sigma_i+z_i-1) \\
&= 4\cdot\inv(\pi') + |\P^N(\pi')| - |\P^C(\pi')| + \sum_{z'_i>0}1 + \sum_{z'_i>0}\big((4\sigma'_i-1)+2(z'_i-1)\big) \\
&= 4\Big(\inv(\pi')+\sum_{z'_i>0}(\sigma'_i+z'_i-1)\Big) + |\P^N(\pi')| - |\P^C(\pi')| - 2\sum_{z'_i>0}(z'_i-1) \\
&= 4\ell(\pi') + |\P^N(\pi')| - |\P^C(\pi')| - 2\sum_{z'_i>0}(z'_i-1).
\end{align*}
By plugging it into \eqref{eq:length}, it follows that
\begin{align}
\chi_{1,b}(\pi) &= (-1)^{4\ell(\pi') + |\P^N(\pi')| - |\P^C(\pi')| - 2\sum_{z'_i>0}(z'_i-1) - 2\sum_{z'_i>0}1} \ \omega^{2b\cdot\col(\pi')} \notag \\
&= (-1)^{|\P^N(\pi')| - |\P^C(\pi')|}\omega^{2b\cdot\col(\pi')} \notag \\
&= (-1)^{|\P(\pi')|}\omega^{2b\cdot\col(\pi')}. \label{eq:chi-hat}
\end{align}
Hence, by Lemma~\ref{lem:fix-color-even}, \eqref{eq:col-hat}, \eqref{eq:fdes-hat}, \eqref{eq:fmaj-hat}, \eqref{eq:chi-hat} and the fact that $\omega^r=1$, the LHS of \eqref{eq:main-even-set} is
\begin{align*}
& \sum_{\pi\in\F_{r,2n}(\H)} \chi_{1,b}(\pi)t^{\fdes(\pi)}q^{\fmaj(\pi)}\prod_{i=1}^{2n}x_i^{z_i} \\
=& \sum_{\pi'\in\widehat{G}_{r,n}(\S)} (-1)^{|\P(\pi')|} \omega^{2b\cdot\col(\pi')} t^{\fdes(\pi')+r|\P(\pi')|} q^{2\fmaj(\pi')+\sum_{i\in\P(\pi')}(2i-1)} \prod_{i=1}^{n}(x_{2i-1}x_{2i})^{z'_i} \\
=& \sum_{\pi'\in\widehat{G}_{r,n}(\S)} \left((-1)^{|\P(\pi')|} t^{r|\P(\pi')|} q^{r\sum_{i\in\P(\pi')}(2i-1)} \right) \omega^{2b\big(r\cdot\maj_F(\pi')+\col(\pi')\big)} t^{\fdes(\pi')} q^{2\fmaj(\pi')} \prod_{i=1}^{n}(x_{2i-1}x_{2i})^{z'_i} \\
=& \sum_{\pi\in G_{r,n}(\S)} \Big(\sum_{A\subseteq[n]}(-1)^{|A|}t^{r|A|}q^{r\sum_{i\in A}(2i-1)}\Big) \omega^{2b\cdot\fmaj(\pi)} t^{\fdes(\pi)} q^{2\fmaj(\pi)} \prod_{i=1}^{n}(x_{2i-1}x_{2i})^{z_i} \\
=& \prod_{i=1}^{n}\big(1-t^rq^{r(2i-1)}\big) \sum_{\pi\in G_{r,n}(\S)} t^{\fdes(\pi)}\big(\omega^b q\big)^{2\fmaj(\pi)}\prod_{i=1}^{n}(x_{2i-1}x_{2i})^{z_i},
\end{align*}
as desired.
\qed

\begin{theorem}\label{thm:main-odd}
Let $r$ and $n$ be two positive integers.
For integer $0\leq b\leq r-1$ and $\S\in\mathfrak{R}_{r,n}$, we have
\begin{equation}
\sum_{\pi\in G_{r,n}(\S)} \chi_{0,b}(\pi)t^{\fdes(\pi)}q^{\fmaj(\pi)}\prod_{i=1}^{n}x_i^{z_i} = \sum_{\pi\in G_{r,n}(\S)}t^{\fdes(\pi)}(\omega^bq)^{\fmaj(\pi)}\prod_{i=1}^{n}x_i^{z_i}. \label{eq:main-odd-set}
\end{equation}
In particular,
\begin{equation}
\sum_{\pi\in G_{r,n}} \chi_{0,b}(\pi)t^{\fdes(\pi)}q^{\fmaj(\pi)}x^{\col(\pi)}  = \sum_{\pi\in G_{r,n}}t^{\fdes(\pi)}(\omega^bq)^{\fmaj(\pi)}x^{\col(\pi)}. \label{eq:main-odd}
\end{equation}
\end{theorem}
\proof
By \eqref{eq:1-dim} and the fact that $\omega^r=1$, the LHS of \eqref{eq:main-odd-set} is equal to
\begin{align*}
& \sum_{\pi\in G_{r,n}(\S)} \omega^{b\cdot\col(\pi)} t^{\fdes(\pi)}q^{\fmaj(\pi)}\prod_{i=1}^{n}x_i^{z_i} \\
=& \sum_{\pi\in G_{r,n}(\S)} \omega^{b(r\cdot\maj_F(\pi)+\col(\pi))} t^{\fdes(\pi)}q^{\fmaj(\pi)}\prod_{i=1}^{n}x_i^{z_i} \\
=& \sum_{\pi\in G_{r,n}(\S)} \omega^{b\cdot\fmaj(\pi)} t^{\fdes(\pi)}q^{\fmaj(\pi)}\prod_{i=1}^{n}x_i^{z_i},
\end{align*}
as desired.
Finally, \eqref{eq:main-odd} can be derived by plugging $\S=(\mathbb{Z}_r,\ldots,\mathbb{Z}_r)$ and $x=x_1=\cdots=x_n$ into \eqref{eq:main-odd-set}.
\qed

\section{Coxeter groups of type $B_n$}\label{sec:Bn}


In $\B_n$ there are four one-dimensional characters: the identity $\chi_{0,0}(\pi)=1$, $\chi_{0,1}(\pi)=(-1)^{\nega(\pi)}$, $\chi_{1,0}(\pi)=(-1)^{\inv(|\pi|)}$ and $\chi_{1,1}(\pi)=(-1)^{\ell_B}(\pi)$, where $|\pi|=|\pi_1||\pi_2|\cdots|\pi_n|$.

Let $\mathfrak{R}_n:=\mathfrak{R}_{2,n}$.
For simplicity, the entries of the $n$-tuple $\S\in\mathfrak{R}_n$ are represented as $+$, $-$ or $\pm$.
For example, $(+,\pm)$ means the first number is positive while the second one can be positive or negative and $\B_2(+,\pm)=\{12,1\bar{2},21,2\bar{1}\}$.

As an application of Theorem~\ref{thm:main-even} and Theorem~\ref{thm:main-odd}, we obtain some sign Euler-Mahonian identities for $\B_n$ as follows.

\begin{corollary}\label{cor:signedEM-B}
Let $n$ be a positive integer.
For $\H\in\mathfrak{R}_{2n}$ we have
\begin{enumerate}
\item[(i)] $\d \sum_{\pi\in\B_{2n}(\H)}(-1)^{\ell_B(\pi)}t^{\fdes(\pi)}q^{\fmaj(\pi)}\prod_{i\in\Nega(\pi)}x_i$\\
\hspace*{1cm} $\d= \prod_{i=1}^{n}(1-t^2q^{4i-2}) \sum_{\pi\in\B_{n}(\S)}t^{\fdes(\pi)}q^{2\fmaj(\pi)}\prod_{i\in\Nega(\pi)}x_{2i-1}x_{2i}$,
\medskip
\item[(ii)] $\d \sum_{\pi\in\B_{2n}(\H)}(-1)^{\inv(|\pi|)}t^{\fdes(\pi)}q^{\fmaj(\pi)}\prod_{i\in\Nega(\pi)}x_i$\\
\hspace*{1cm} $\d= \prod_{i=1}^{n}(1-t^2q^{4i-2}) \sum_{\pi\in\B_{n}(\S)}t^{\fdes(\pi)}q^{2\fmaj(\pi)}\prod_{i\in\Nega(\pi)}x_{2i-1}x_{2i}$,
\end{enumerate}
where $\S\in\mathfrak{R}_n$ is the unique $n$-tuple such that $\S\prec\H$;
and, for $\S\in\mathfrak{R}_n$ we have
\begin{enumerate}
\item[(iii)] $\d \sum_{\pi\in\B_{n}(\S)}(-1)^{\nega(\pi)}t^{\fdes(\pi)}q^{\fmaj(\pi)}\prod_{i\in\Nega(\pi)}x_i = \sum_{\pi\in\B_{n}(\S)}t^{\fdes(\pi)}(-q)^{\fmaj(\pi)}\prod_{i\in\Nega(\pi)}x_i$.
\end{enumerate}
In particular,
\begin{enumerate}
\item[(i)] $\d \sum_{\pi\in\B_{2n}}(-1)^{\ell_B(\pi)}t^{\fdes(\pi)}q^{\fmaj(\pi)}x^{\nega(\pi)} = \prod_{i=1}^{n}(1-t^2q^{4i-2}) \sum_{\pi\in\B_{n}}t^{\fdes(\pi)}q^{2\fmaj(\pi)}x^{2\nega(\pi)}$,
\item[(ii)] $\d \sum_{\pi\in\B_{2n}}(-1)^{\inv(|\pi|)}t^{\fdes(\pi)}q^{\fmaj(\pi)}x^{\nega(\pi)} = \prod_{i=1}^{n}(1-t^2q^{4i-2}) \sum_{\pi\in\B_{n}}t^{\fdes(\pi)}q^{2\fmaj(\pi)}x^{2\nega(\pi)}$,
\item[(iii)] $\d \sum_{\pi\in\B_{n}}(-1)^{\nega(\pi)}t^{\fdes(\pi)}q^{\fmaj(\pi)}x^{\nega(\pi)} = \sum_{\pi\in\B_{n}}t^{\fdes(\pi)}(-q)^{\fmaj(\pi)}x^{\nega(\pi)}$.
\end{enumerate}
\end{corollary}

In what follows, we aim to derive the signed Eulerian identities for $B_{2n+1}$, which can be seen as natural generalizations of \eqref{eq:signedE-A-odd}.

For a given $\H\in\mathfrak{R}_{2n+1}$, consider the involution $\varphi$ on $\B_{2n+1}(\H)$ which is defined on $G_{r,2n}(\H)$ in Section~\ref{sec:Grn-main}.
That is, $\varphi(\pi)$ is obtained from $\pi$ by swapping the two letters $2i-1$ and $2i$, where $i$ is the smallest integer such that $2i-1$ and $2i$ have opposite signs or are not in adjacent positions.
Similar to the same arguments in Section~\ref{sec:Grn-main}, the fixed points under $\varphi$ are those signed permutations in which the letters $2i-1$ and $2i$ are adjacent and having the same sign, for $i\in[n]$.
Denote by $\mathcal{F}_{2n+1}(\H)$ the set of fixed points.
By the same argument in the proof of Lemma~\ref{lem:fix-color-even}, when $\pi$ is not a fixed point, $(-1)^{\inv(|\pi|)}=-(-1)^{\inv(|\varphi(\pi)|)}$ and $\fdes(\pi)=\fdes(\varphi(\pi))$.
It follows that
\begin{equation}\label{eq:fix-B-absinv}
\sum_{\pi\in\B_{2n+1}(\H)}(-1)^{\inv(|\pi|)}t^{\fdes(\pi)} = \sum_{\pi\in\mathcal{F}_{2n+1}(\H)}(-1)^{\inv(|\pi|)}t^{\fdes(\pi)}.
\end{equation}

For $k\in[n]$ let $\S_k$ be the $n$-tuple $\S\in\mathfrak{R}_n$ with an extra tilde mark above the $k$-th entry.
For $\S_k=(S_1,\ldots,\tilde{S}_k,\ldots,S_n)$, define $\B_n(\S_k)\subseteq\B_n(\S)$ to be the set of signed permutations such that (i) $\pi_i/|\pi_i|\in S_i$ for $1\leq i\leq n$ and (ii) $|\pi_k|=n$.

Consider $\H\in\mathfrak{R}_{2n+1}$ and $\S\in\mathfrak{R}_{n+1}$.
We say $\S_k\triangleleft\H$ if
\begin{align*}
S_i = \begin{cases}
H_{2i-1}\cap H_{2i}, & \text{if }i<k, \\
H_{2i-1}, & \text{if }i=k, \\
H_{2i}\cap H_{2i+1}, & \text{if }i>k.
\end{cases}
\end{align*}
Denote by $\mathbb{F}(\H)$ be the collection of $(n+1)$-tuples $\S_k$ such that $\S_k\triangleleft\H$.
For example, if $\H=(\pm,-,+,\pm,\pm)$, then $\mathbb{F}(\H)=\{(\tilde{\pm},\emptyset,\pm),(-,\tilde{+},\pm),(-,+,\tilde{\pm})\}$, and $\B_3((\tilde{\pm},\emptyset,\pm))=\emptyset$, $\B_3((-,\tilde{+},\pm))=\{\bar{1}32,\bar{1}3\bar{2},\bar{2}31,\bar{2}3\bar{1}\}$, $\B_3((-,+,\tilde{\pm}))=\{\bar{1}23,\bar{1}2\bar{3},\bar{2}13,\bar{2}1\bar{3}\}$.
Obviously, $\B_n(\S_h)\cap\B_n(\S_k)=\emptyset$ whenever $h\neq k$.

\begin{theorem}\label{thm:B-odd-absinv}
Let $n$ be a positive integer.
For $\H\in\mathfrak{R}_{2n+1}$, we have
\begin{align*}
\sum_{\pi\in\B_{2n+1}(\H)}(-1)^{\inv(|\pi|)}t^{\fdes(\pi)} = (1-t^2)^n \sum_{\S_k\in\mathbb{F}(\H)}\sum_{\pi\in\B_{n+1}(\S_k)}t^{\fdes(\pi)}.
\end{align*}
In particular,
\begin{align*}
\sum_{\pi\in\B_{2n+1}}(-1)^{\inv(|\pi|)}t^{\fdes(\pi)} = (1-t^2)^n \sum_{\pi\in\B_{n+1}}t^{\fdes(\pi)}.
\end{align*}
\end{theorem}
\proof
The second identity can be obtained from the first one by letting $\H=(\pm,\ldots,\pm)$.
More precisely, when $\H=(\pm,\ldots,\pm)$, $\B_{2n+1}(\H)=\B_{2n+1}$ and $\mathbb{F}(\H)=\S_1\uplus\S_2\uplus\cdots\uplus\S_{n+1}$, where $\S=(\pm,\ldots,\pm)\in\mathfrak{R}_{n+1}$.
Therefore, $$\biguplus_{\S_k\in\mathbb{F}(\H)}\B_{n+1}(\S_k) = \B_{n+1}.$$

By \eqref{eq:fix-B-absinv}, to derive the first identity it suffices to consider the set $\mathcal{F}_{2n+1}(\H)$.
Consider the bijective correspondence $\phi$ between elements in $\mathcal{F}_{2n+1}(\H)$ and elements in $\uplus_{\S_k\in\mathbb{F}(\H)}\B_{n+1}(\S_k)$ with some letters marked a ``hat'', according to the following rules.
\begin{itemize}
\item Each pair of adjacent entries of type $\pm(2j-1),\pm 2j$ in $\mathcal{F}_{2n+1}(\H)$ is replaced by $\pm j$.
\item Each pair of adjacent entries of type $\pm 2j,\pm(2j-1)$ in $\mathcal{F}_{2n+1}(\H)$ is replaced by $\pm\hat{j}$.
\item The entry $\pm(2n+1)$ in $\mathcal{F}_{2n+1}(\H)$ is replaced by $\pm(n+1)$.
\end{itemize}
We denote by $\widehat{\B}_{n+1}(\S_k)$ the set of all hatted $\S_k$-restricted signed permutations on $[n+1]$ for $\S_k\in\mathbb{F}(\H)$.
For $\pi'\in\widehat{\B}_{n+1}(\S_k)$, let $\inv(|\pi'|)$, $\Des_F(\pi')$ and $\fdes(\pi')$ be defined by omitting the hats, and let $\L(\pi')$ and $\P(\pi')$ respectively be the sets of the hatted letters and their positions.
Note here that $0\leq|\L(\pi')|=|\P(\pi')|\leq n$, $k\notin\P(\pi')$, and $n+1,\overline{n+1}\notin\L(\pi')$.
For example, let $\pi=\bar{5}\bar{6}\bar{2}\bar{1}879\bar{4}\bar{3}$.
Then $\pi'=\phi(\pi)=\bar{3}\hat{\bar{1}}\hat{4}5\hat{\bar{2}}$ is the corresponding hatted signed permutation with $\inv(|\pi'|)=4$, $\Des_F(\pi')=\{1,4\}$, $\fdes(\pi')=5$, $\L(\pi')=\{1,2,4\}$ and $\P(\pi')=\{2,3,5\}$.

Pick $\pi\in\mathcal{F}_{2n+1}(\H)$.
Denote by $\pi'=\phi(\pi)\in\widehat{\B}_{n+1}(\S_k)$ for some $\S_k\in\mathbb{F}(\H)$.
Let $\max(\pi)$ be the position of letter $\pm(2n+1)$ in $\pi$.
Obviously, $\max(\pi)=2k-1$.
We further let $\sigma$ (resp., $\sigma'$) be the resulting $\H$-restricted signed permutation (resp., hatted $\S_k$-restricted signed permutation) obtained from $\pi$ (resp., $\pi'$) by removing the letter $\pm(2n+1)$ (resp., $\pm(n+1)$).
By the definition of $\phi$, we have
\begin{align*}
\inv(|\pi|) &= \inv(|\sigma|) + \big(2n+1-\max(\pi)\big) \\
&= 4\inv(|\sigma'|) + |\L(\sigma')| + \big(2n+1-(2k-1)\big) \\
&= 4\inv(|\sigma'|) + |\L(\pi')| + 2\big(n-k-1\big),
\end{align*}
which implies that
\begin{equation}\label{eq:B-odd-absinv-fold1}
(-1)^{\inv|\pi|} = (-1)^{|\L(\pi')|}.
\end{equation}
Moreover,
\begin{align*}
\Des_F(\pi) =& \{2i:\,i\in\Des_F(\pi'),\,i<k\} \cup \{2i-1:\,i\in\Des_F(\pi'),\,i\geq k\} \\
&\cup \{2i:\,i\in\P(\pi'),i<k\} \cup \{2i-1:\,i\in\P(\pi'),\,i>k\}.
\end{align*}
Since $\pi_1$ and $\pi'_1$ have the same sign, it follows that
\begin{equation}\label{eq:B-odd-absinv-fold2}
\fdes(\pi) = 2\big(\des_F(\pi')+|\P(\pi')|\big) + \delta(\pi_1<1) = \fdes(\pi')+2|\L(\pi')|.
\end{equation}
Hence, by \eqref{eq:B-odd-absinv-fold1} and \eqref{eq:B-odd-absinv-fold2} we have
\begin{align*}
\sum_{\pi\in\mathcal{F}_{2n+1}(\H)}(-1)^{\inv(|\pi|)}t^{\fdes(\pi)}
=& \sum_{\S_k\in\mathbb{F}(\H)}\sum_{\pi'\in\widehat{\B}_{n+1}(\S_k)} \left((-1)^{|\L(\pi')|}t^{2|\L(\pi')|}\right) t^{\fdes(\pi')} \\
=& \sum_{\S_k\in\mathbb{F}(\H)}\sum_{\pi'\in\B_{n+1}(\S_k)} \left(\sum_{A\subseteq[n]}(-1)^{|A|}t^{2|A|}\right) t^{\fdes(\pi')} \\
=& (1-t^2)^n \sum_{\S_k\in\mathbb{F}(\H)}\sum_{\pi'\in\B_{n+1}(\S_k)} t^{\fdes(\pi')}.
\end{align*}
This completes the proof.
\qed

\medskip

Now, we consider the signed identity in which the sign is taken to be $(-1)^{\ell_B}$.
Recall that $\varphi$ is an involution on $\B_{2n+1}$.
Similar to the argument for getting the identity \eqref{eq:fix-B-absinv}, as $(-1)^{\ell_B(\pi)}=-(-1)^{\ell_B(\varphi(\pi))}$ for any non-fixed point under $\varphi$ on $\B_{2n+1}$, we have
\begin{equation}\label{eq:fix-B-length}
\sum_{\pi\in\B_{2n+1}}(-1)^{\ell_B(\pi)}t^{\fdes(\pi)} = \sum_{\pi\in\mathcal{F}_{2n+1}}(-1)^{\ell_B(\pi)}t^{\fdes(\pi)},
\end{equation}
where $\mathcal{F}_{2n+1}$ simply is the set of fixed points under $\varphi$ on $\B_{2n+1}$.
Note that $\mathcal{F}_{2n+1}$ consists of those signed permutations in which the letters $2i-1$ and $2i$ are adjacent and having the same sign, for $i\in[n]$.

For a signed permutation $\pi\in\B_n$ let $$\Des_B(\pi):=\{i:\,0\leq i\leq n-1,\pi_i>\pi_{i+1}\}$$
with respect to the natural linear order $$\bar{n}<\cdots<\bar{1}<1\cdots<n,$$ where $\pi_0:=0$ by convention.
Following \cite{BB_05}, the \emph{type $B$ descent number} of $\pi$, denoted by $\des_B(\pi)$, is then defined to be the cardinality of $\Des_B(\pi)$, and the \emph{type B major index} of $\pi$ is given by $$\maj_B(\pi):=\sum_{i\in\Des_B(\pi)}i.$$
We use $\sgm$ to record the \emph{sign of the maximum letter} (i.e., $n$) in $\pi$ by letting $\sgm(\pi)=1$ if its sign is negative and $0$ otherwise.
For example, $\sgm(\bar{2}51\bar{3}4)=0$ and $\sgm(\bar{2}\bar{5}1\bar{3}4)=1$.

\begin{theorem}\label{thm:B-odd-length}
For any positive integer $n$, we have
\begin{align*}
\sum_{\pi\in\B_{2n+1}}(-1)^{\ell_B(\pi)}t^{\fdes(\pi)} = (1-t^2)^n(1-t)\sum_{\pi\in\B_n}t^{2\des_B(\pi)}.
\end{align*}
\end{theorem}
\begin{proof}
Consider the bijection $\phi$ given in the proof of Theorem~\ref{thm:B-odd-absinv} between the elements in $\F_{2n+1}$ and the hatted signed permutations in $\widehat{\B}_{n+1}$.
Here, we view $\widehat{\B}_{n+1}$ as $\uplus_{k=1}^{n+1}\widehat{\B}_{n+1}(\S_k)$, where $\S=(\pm,\ldots,\pm)\in\mathfrak{R}_{n+1}$.
Pick $\pi\in\F_{2n+1}$.
Denote by $\pi'=\phi(\pi)\in\widehat{\B}_{n+1}$.
The definitions of $\max(\pi)$, $\inv(\pi')$, $\ell_B(\pi')$, $\Des_F(\pi')$, $\fdes(\pi')$, $\L(\pi')$ and $\P(\pi')$ are of the same as before.
Let $\sigma$ (resp.~$\sigma'$) be the resulting signed permutation (resp.~hatted signed permutation) obtained from $\pi$ (resp.~$\pi'$) by removing the letter $\pm(2n+1)$ (resp.~$\pm(n+1)$).
Observe that, if $\max(\pi)=2k-1$, then $\pi'\in\widehat{\B}_{n+1}(\S_k)$.

We consider the sign of the letter $2n+1$ in $\pi$.
First, if $\sgm(\pi)=0$, then
\begin{align*}
\ell_B(\pi) = \ell_B(\sigma) + 2n+1-\max(\pi).
\end{align*}
By the definition of $\ell_B$ and \eqref{eq:inv-hat}, we have
\begin{align*}
\ell_B(\pi) &= 4\inv(\sigma') + |\P^{+}(\sigma')| + \nega(\sigma') - |\P^{-}(\sigma')| - \sum_{i\in\Nega(\sigma')}(4\sigma'_i+1) +2n+1-\max(\pi) \\
&= 4\inv(\sigma') + |\P(\sigma')| - 2|\P^{-}(\sigma')| - 4\biggl(\sum_{i\in\Nega(\sigma')}\sigma'_i\biggr) + 2(n-k+1),
\end{align*}
where $\P^{+}(\sigma'):=\{i\in\P(\sigma'):\,\sigma'_i>0\}$ and $\P^{-}(\sigma'):=\{i\in\P(\sigma'):\,\sigma'_i<0\}$.
Since $|\P(\sigma')|=|\P(\pi')|$, it follows that $(-1)^{\ell_B(\pi)}=(-1)^{|\P(\pi')|}$.
On the other hand, if $\sgm(\pi)=1$, then
\begin{align*}
\ell_B(\pi) = \ell_B(\sigma) + 2n+\max(\pi).
\end{align*}
By the same argument, we have
\begin{align*}
\ell_B(\pi) = 4\inv(\sigma') + |\P(\sigma')| - 2|\P^{-}(\sigma')| - 4\biggl(\sum_{i\in\Nega(\sigma')}\sigma'_i\biggr)+ 2(n+k)-1,
\end{align*}
which implies that $(-1)^{\ell_B(\pi)}=(-1)^{|\P(\pi')|+1}$.
Since $\sgm(\pi)=\sgm(\pi')$, we conclude that
\begin{equation}\label{eq:B-odd-length-fold1}
(-1)^{\ell_B(\pi)}=(-1)^{|\P(\pi')|+\mathsf{sgm}(\pi')}.
\end{equation}
Moreover, as shown in \eqref{eq:B-odd-absinv-fold2}, we replace $|\L(\pi')|$ with $|\P(\pi')|$ and then obtain
\begin{equation}\label{eq:B-odd-length-fold2}
\fdes(\pi) = \fdes(\pi')+2|\P(\pi')|.
\end{equation}
Hence, by \eqref{eq:B-odd-length-fold1} and \eqref{eq:B-odd-length-fold2} it follows that
\begin{align*}
 \sum_{\pi\in\mathcal{F}_{2n+1}}(-1)^{\ell_B(\pi)}t^{\fdes(\pi)}
=& \sum_{\pi'\in\widehat{\B}_{n+1}} \biggl((-1)^{|\P(\pi')|}t^{2|\P(\pi')|}\biggr) (-1)^{\sgm(\pi')} t^{\fdes(\pi')} \\
=& \sum_{\pi'\in\B_{n+1}} \biggl(\sum_{A\subseteq[n]}(-1)^{|A|}t^{2|A|}\biggr) (-1)^{\sgm(\pi')} t^{\fdes(\pi')} \\
=& (1-t^2)^n \sum_{\pi'\in\B_{n+1}} (-1)^{\sgm(\pi')} t^{\fdes(\pi')},
\end{align*}
which completes the proof in view of the Lemma~\ref{lem:lin} below.
\end{proof}

\begin{lemma}\label{lem:lin}
For $n\geq1$, we have
\begin{equation}\label{eq:lin}
\sum_{\pi\in\B_{n+1}} (-1)^{\sgm(\pi)} t^{\fdes(\pi)}= (1-t)\sum_{\pi\in\B_n}t^{2\des_B(\pi)}.
\end{equation}
\end{lemma}

The proof is based on the technique of  flag barred permutations developed by the second named author~\cite{Lin}, which was originally inspired by Gessel and Stanley~\cite{GS}.
We need some preparations.

For any signed permutation $\pi\in\B_n$, define the {\em sign change function} $\delta^{\pi}: [n]\rightarrow \{0,1\}^n$ as
$$
\delta^{\pi}(i)=\delta_i^{\pi}:=1\text{ if $\pi_i\pi_{i+1}<0$ or $i=n$ and $\pi_n<0$, otherwise $0$}.
$$
Let $\ch(\pi):=\sum_i\delta_i^{\pi}$ be the {\em total sign change} of $\pi$. For instance, for $\pi=\bar{1}\bar{3}42\bar{5}\in\B_5$, we have $\delta^{\pi}=01011$ and $\ch(\pi)=3$.
It was shown in~\cite[Lemma~17]{Lin} that
\begin{equation}\label{eq:ch}
\fdes(\pi)=\ch(\pi)+\sum_{i\in\Des_F(\pi)\atop\delta_i^{\pi}=0}2.
\end{equation}
This relationship motivates the definition of flag barred permutations below.

Let $\pi\in\B_n$.
We call the space between $\pi_i$ and $\pi_{i+1}$ the $i$-th space of $\pi$ for $1\leq i\leq n-1$ and it is called a descent space if $i\in\Des_F(\pi)$.
We also call the space befor $\pi_1$ and the space after $\pi_n$ the $0$-th space and the $n$-th space of $\pi$, respectively.
A {\em flag barred permutation} on $\pi$ is obtained from $\pi$ by inserting bars such that
\begin{itemize}
\item for each $i$ ($1\leq i\leq n$), the $i$-th descent space of $\pi$ with $\delta_i^{\pi}=0$ receives at least $2$ bars;
\item the parity of the number of bars in the $i$-th space ($1\leq i\leq n$) of $\pi$ has the same parity as $\delta_i^{\pi}$.
\end{itemize}
In view of relationship~\eqref{eq:ch}, every barred permutation on $\pi$  has at least $\fdes(\pi)$ bars.
For example, if $\pi=\bar{2}\bar{3}1\bar{5}\bar{4}\in\B_5$, then the flag barred permutation on $\pi$ with the least number of bars is $\bar{2}\bar{3}\vert1\vert\bar{5}\vert\vert\bar{4}\vert$.

We are now ready for the proof of Lemma~\ref{lem:lin}.

\begin{proof}[{\bf Proof of Lemma~\ref{lem:lin}}]
Let $\B_{n+1}^+$ (resp.~$\B_{n+1}^-$) denote the set of signed permutations $\pi\in\B_{n+1}$ such that $\sgm(\pi)=0$ (resp.~$\sgm(\pi)=1$), namely, $n+1$ (resp.~$\overline{n+1}$) appears as a letter of $\pi$. Let $\widetilde{\B}_{n+1}^+$ (resp.~$\widetilde{\B}_{n+1}^-$) be the set of flag barred permutations on $\B_{n+1}^+$ (resp.~$\B_{n+1}^-$).

For each flag barred permutation $\tilde{\pi}$, the weight  of $\tilde{\pi}$, denoted $\wt(\tilde{\pi})$, is defined by $\wt(\tilde{\pi})=t^{\bars(\tilde{\pi})}$, where  $\bars(\tilde{\pi})$ is the number of bars in $\tilde{\pi}$.  We are going to count the flag barred permutations in $\widetilde{\B}_{n+1}^+$ by the weight ``$\wt$'' in two different ways. First, fix a signed permutation $\pi\in\B_{n+1}^+$, and sum over all flag barred permutations on $\pi$. Then, fix the number of bars $k$, and sum over all flag barred permutations with $k$ bars.

Fix a permutation $\pi\in\B_{n+1}^+$.  The flag barred permutation on $\pi$ with the least number of bars, denoted by $\bar{\pi}$, has the weight $t^{\fdes(\pi)}$.  As every  flag barred permutation on $\pi$ can be obtained from $\bar{\pi}$ by further inserting  any  number of bars in the $0$-th space of $\pi$ and an even number of bars in the $i$-th ($1\leq i\leq n+1$) space of $\pi$, we see that counting all the flag barred permutations on $\pi$ according to the weight ``$\wt$'' gives
$$
t^{\fdes(\pi)}(1+t+t^2+\cdots)(1+t^2+t^4+\cdots)^{n+1}=\frac{t^{\fdes(\pi)}}{(1-t)(1-t^2)^{n+1}}.
$$
Therefore, we have
\begin{equation}\label{posi}
\sum_{\tilde{\pi}\in\widetilde{\B}_{n+1}^+}{\wt(\tilde{\pi})}=\frac{\sum_{\pi\in\B_{n+1}^+}t^{\fdes(\pi)}}{(1-t)(1-t^2)^{n+1}}.
\end{equation}

For a fixed integer $k\geq0$, let $\widetilde{\B}_{n+1,k}^+$ be the set of flag barred permutations in $\widetilde{\B}_{n+1}^+$ with $k$ bars. Now, each flag barred permutation from $\widetilde{\B}_{n+1,k}^+$ can be obtained  in two steps. First, we put $k$ bars in one line and insert the letters $1,2,\ldots,n+1$ into the spaces between bars. Second, we determine the signs of all the letters and the order of the letters between each pair of two adjacency bars in the resulting object in such a way that  it becomes a flag barred permutation. The signs and the orders of letters are unique according to the definition of a flag barred permutation, which requires
\begin{itemize}
\item [(a)] all letters between two adjacency bars have the same sign and are in increasing order with respect to the linear order
$$
\bar{1}\cdots<\bar{n}<\overline{n+1}<1<\cdots<n<n+1.
$$
\item [(b)] the sign of an integer in the $(i+1)$-th space (from right to left) of the $k+1$ spaces of the $k$ bars is ``$+$'' (resp.~``$-$'') if $i$ is even (resp.~odd).
\end{itemize}
In general, there are $k+1$ ways to insert the letter $i$ for  $1\leq i\leq n$ and because of rule (b) above, the letter $n+1$ must be inserted in the odd numbered space (from right to left) of the $k+1$ spaces of the $k$ bars in order to have positive sign. Thus,
$$\sum_{\tilde{\pi}\in\widetilde{\B}_{n+1,k}^+}\wt(\pi)=(k+1)^n\lceil(k+1)/2\rceil t^k$$
and hence
$$
\sum_{\tilde{\pi}\in\widetilde{\B}_{n+1}^+}{\wt(\tilde{\pi})}=\sum_{k\geq0}\sum_{\tilde{\pi}\in\widetilde{\B}_{n+1,k}^+}\wt(\pi)=\sum_{k\geq0}(k+1)^n\lceil(k+1)/2\rceil t^k.
$$
Comparing with~\eqref{posi} we get the identity
\begin{equation}\label{posi2}
\frac{\sum_{\pi\in\B_{n+1}^+}t^{\fdes(\pi)}}{(1-t)(1-t^2)^{n+1}}=\sum_{k\geq0}(k+1)^n\lceil(k+1)/2\rceil t^k.
\end{equation}
Similarly, applying the same approach to  $\widetilde{\B}_{n+1}^-$ results in
\begin{equation}\label{nega2}
\frac{\sum_{\pi\in\B_{n+1}^-}t^{\fdes(\pi)}}{(1-t)(1-t^2)^{n+1}}=\sum_{k\geq0}(k+1)^n\lfloor(k+1)/2\rfloor t^k.
\end{equation}
Combining~\eqref{posi2} and~\eqref{nega2} yields that
\begin{equation}\label{eq:nepo}
\frac{\sum_{\pi\in\B_{n+1}} (-1)^{\sgm(\pi)} t^{\fdes(\pi)}}{(1-t)(1-t^2)^{n+1}}=\sum_{k\geq0}(2k+1)^nt^{2k}.
\end{equation}

On the other hand, it is well known (see Theorem~3.4 in~\cite{Brenti} with $q=1$) that
$$
\frac{\sum_{\pi\in\B_n}t^{\des_B(\pi)}}{(1-t)^{n+1}}=\sum_{k\geq0}(2k+1)^nt^k.
$$
It then follows that
$$
\frac{(1-t)\sum_{\pi\in\B_n}t^{2\des_B(\pi)}}{(1-t)(1-t^2)^{n+1}}=\sum_{k\geq0}(2k+1)^nt^{2k},
$$
which gives~\eqref{eq:lin} in view of~\eqref{eq:nepo}
\end{proof}

\section{Coxeter groups of type $D_n$}\label{sec:Dn}

The \emph{even-signed permutation group} $\D_n$ is the subgroup of $\B_n$ defined by
$$\D_n:=\{\pi\in\B_n:\, \nega(\pi)\text{ is even}\},$$
which consists of those signed permutations with an even number of negative entries.
The group $\D_n$ is known as the Coxeter group of type $D_n$ which has the generators $s_0',s_1,\ldots,s_{n-1}$, where $s_0'=(\bar{1},2)$ and $s_i=(i,i+1)$ for $i\geq 1$.
Let $\ell_D$ be the corresponding length function of $\D_n$.
It is known~\cite{BB_05} that the combinatorial description of $\ell_D$ is
$$\ell_D(\pi)=\inv(\pi)-\sum_{i\in\Nega(\pi)}(\pi_i+1).$$
For example, $\ell_D(\bar{2}3\bar{5}\bar{1}\bar{4})=6-(-8)=14$.

For $\pi\in\D_n$ the \emph{D-descent number} and the \emph{D-major index} of $\pi$ are respectively defined in~\cite{BC_04} by
\begin{align*}
\ddes(\pi) &:= \fdes(\pi_1\pi_2\cdots\pi_{n-1}|\pi_n|), \\
\dmaj(\pi) &:= \fmaj(\pi_1\pi_2\cdots\pi_{n-1}|\pi_n|),
\end{align*}
where $\fdes$ and $\fmaj$ are considered in $\B_n=G(2,n)$.
For example, if $\pi=\bar{2}3\bar{5}\bar{1}\bar{4}$,
then $\ddes(\pi)=\fdes(\bar{2}3\bar{5}\bar{1}4)=5$ and $\dmaj(\pi)=\fmaj(\bar{2}3\bar{5}\bar{1}4)=13$.

%

As in $\D_n$ there are only two one-dimensional characters, $1$ and $(-1)^{\ell_D(\pi)}$, it suffices to consider the case when $\chi(\pi)=(-1)^{\ell_D(\pi)}$.
Let us first consider the case for $n$ being even.

For $\pi\in\D_{2n}$ let $i$ be the smallest integer such that the letters $2i-1$ and $2i$ satisfy one of the following conditions.
\begin{enumerate}
\item[(A1)] They are not in adjacent positions.
\item[(A2)] They have opposite signs, and are not both at the last two positions.
\item[(A3)] They are both at the last two positions with negative signs.
\end{enumerate}
Then, let $\eta(\pi)$ be the even-signed permutation obtained from $\pi$ by swapping the two letters $2i-1$ and $2i$.
For examples, $\eta(21\bar{3}\bar{5}6\bar{4})=21\bar{4}\bar{5}6\bar{4}$, $\eta(21\bar{5}63\bar{4})=21\bar{6}53\bar{4}$ and $\eta(21\bar{3}\bar{4}\bar{5}\bar{6})=21\bar{3}\bar{4}\bar{6}\bar{5}$.
Clearly, $\eta$ is an involution on $\D_{2n}$ whose fixed points are those even-signed permutations $\pi$ in which the letters $2i-1$ and $2i$ are adjacent and having the same sign, and both $\pi_{2n-1}$ and $\pi_{2n}$ are positive.
Note that the last property ``both $\pi_{2n-1}$ and $\pi_{2n}$ are positive'' is because of even number of negatives.
Denote by $\mathcal{F}_{2n}^D$ the set of fixed points under the  involution $\eta$.

Let $\pi$ be a non-fixed point under $\eta$ and denote by $\pi'=\eta(\pi)$.
Obviously, $\Nega(\pi)=\Nega(\pi')$.
It is also clear that
\begin{align*}
\Des_F(\pi_1\pi_2\cdots\pi_{2n-1}|\pi_{2n}|) = \Des_F(\pi'_1\pi'_2\cdots\pi'_{2n-1}|\pi'_{2n}|),
\end{align*}
which implies that $\ddes(\pi)=\ddes(\pi')$ and $\dmaj(\pi)=\dmaj(\pi')$.
Assume $i$ is the smallest integer such that $2i-1$ and $2i$ satisfy (A1), (A2) or (A3).
If the two letters $2i-1$ and $2i$ are of the same sign, then
\begin{align*}
\inv(\pi) = \inv(\pi')\pm 1 \quad \text{and} \quad \sum_{i\in\Nega(\pi)}(\pi_i+1) =\sum_{i\in\Nega(\pi')}(\pi'_i+1);
\end{align*}
otherwise, if they are of opposite signs, then
\begin{align*}
\inv(\pi) = \inv(\pi') \quad \text{and} \quad \sum_{i\in\Nega(\pi)}(\pi_i+1) =\left(\sum_{i\in\Nega(\pi')}(\pi'_i+1)\right)\pm 1.
\end{align*}
In either case, it follows that $\ell_D(\pi)=\ell_D(\pi')\pm 1$.
Therefore, we conclude that
\begin{equation}\label{eq:D-ell-fix}
\sum_{\pi\in\D_{2n}} (-1)^{\ell_D} t^{\ddes(\pi)}q^{\dmaj(\pi)}\prod_{i\in\Nega(\pi)}x_i = \sum_{\pi\in\mathcal{F}_{2n}^D} (-1)^{\ell_D} t^{\ddes(\pi)}q^{\dmaj(\pi)}\prod_{i\in\Nega(\pi)}x_i.
\end{equation}

We are ready to derive the following result.

\begin{theorem}\label{thm:signEM-D}
For any positive integer $n$ we have
\begin{align*}
\sum_{\pi\in\D_{2n}} &(-1)^{\ell_D} t^{\ddes(\pi)}q^{\dmaj(\pi)}\prod_{i\in\Nega(\pi)}x_i \\
&= \prod_{i=1}^{n} (1-t^2q^{4i-2}) \sum_{\pi\in\D_{n}} t^{\ddes(\pi)}q^{2\dmaj(\pi)}\prod_{i\in\Nega(\pi)\setminus\{n\}}x_{2i-1}x_{2i}.
\end{align*}
\end{theorem}
\proof
By \eqref{eq:D-ell-fix} it suffices to consider the set $\mathcal{F}_{2n}^D$. Introduce the bijective correspondence $\phi$ between elements in $\mathcal{F}_{2n}^D$ and elements in $\D_n$ with some letters marked a ``hat'', according to the following rules.
\begin{itemize}
\item Each pair of adjacent entries of type $\pm(2j-1),\pm 2j$ in $\mathcal{F}_{2n}^D$ is replaced by $\pm j$.
\item Each pair of adjacent entries of type $\pm 2j,\pm(2j-1)$ in $\mathcal{F}_{2n}^D$ is replaced by $\pm\hat{j}$.
\item After the above two steps, if the number of negatives of the resulting permutation is odd, then change the sign of the last entry from positive to negative.
\end{itemize}
Note that the last rule is well-defined, since after the first two steps, the last entry must be positive due to $\pi_{2n-1}>0$ and $\pi_{2n}>0$ for $\pi\in\mathcal{F}_{2n}^D$.
We denote by $\widehat{\D}_{n}$ the set of all hatted even-signed permutations on $[n]$.
For $\pi'\in\widehat{\D}_{n}$ let $\inv(\pi')$, $\Nega(\pi')$, $\nega(\pi')$, $\ddes(\pi')$ and $\dmaj(\pi')$ be defined by omitting the hats, and let $\P(\pi')$ be the set of the positions of the hatted letters.
For example, $\phi(21\bar{5}\bar{6}8743)=\hat{1}\bar{3}\hat{4}\hat{\bar{2}}$.

Pick $\pi\in\mathcal{F}_{2n}^D$.
Denote by $\pi'=\phi(\pi)$ and $\sigma'=\pi'_1\cdots\pi'_{n-1}|\pi'_n|$.
By the same argument for getting the identity \eqref{eq:inv-hat}, we have
\begin{align*}
\inv(\pi)=4\inv(\sigma')+|\P^{+}(\sigma')|-|\P^{-}(\sigma')|+\nega(\sigma'),
\end{align*}
where $\P^{+}(\sigma'):=\{i\in\P(\sigma'):\,\sigma'_i>0\}$ and $\P^{-}(\sigma'):=\{i\in\P(\sigma'):\,\sigma'_i<0\}$.
By the definition of $\ell_D$, it follows that
\begin{align*}
\ell_D(\pi) &= 4\inv(\sigma')+|\P^{+}(\sigma')|-|\P^{-}(\sigma')|+\nega(\sigma') - \sum_{\sigma'_i<0}\left((4\sigma'_i+1)+2\right) \\
&= 4\inv(\sigma') + |\P(\sigma')| - 2|\P^{-}(\sigma')| - 2\nega(\sigma') - 4\sum_{\sigma'_i<0}\sigma'_i,
\end{align*}
which implies that
\begin{equation}\label{eq:D-even-length-fold1}
(-1)^{\ell_D(\pi)} = (-1)^{|\P(\pi')|}
\end{equation}
due to the fact that $\P(\pi')=\P(\sigma')$.
By the same arguments for getting the identities \eqref{eq:DesF-hat}--\eqref{eq:fmaj-hat}, we have
\begin{align*}
\Des_F(\pi) = \{2i:\,i\in\Des_F(\sigma')\} \cup \{2i-1:\,i\in\P(\sigma')\},
\end{align*}
and thus
\begin{align}
\ddes(\pi) &= \fdes(\pi) = \fdes(\sigma') + 2|\P(\sigma')| = \ddes(\pi') + 2|\P(\pi')| \label{eq:D-even-length-fold2} \\
\dmaj(\pi) &= \fmaj(\pi) = 2\fmaj(\sigma') + 2\sum_{i\in\P(\sigma')}(2i-1) = 2\dmaj(\pi') + \sum_{i\in\P(\pi')}(4i-2), \label{eq:D-even-length-fold3}
\end{align}
where the first equalities of the two identities are due to $\pi_{2n}>0$.
Furthermore, it is easy to see that
\begin{equation}\label{eq:D-even-length-fold4}
\Nega(\pi) = \biguplus_{i\in\Nega(\pi')\setminus\{n\}} \{2i-1,2i\}.
\end{equation}
Combining \eqref{eq:D-even-length-fold1}--\eqref{eq:D-even-length-fold4} yields that
\begin{align*}
& \sum_{\pi\in\mathcal{F}_{2n}^D} (-1)^{\ell_D} t^{\ddes(\pi)}q^{\dmaj(\pi)}\prod_{i\in\Nega(\pi)}x_i \\
=& \sum_{\pi'\in\widehat{\D}_{n}} \left((-1)^{|\P(\pi')|}t^{2|\P(\pi')|}q^{\sum_{i\in\P(\pi')}(4i-2)}\right) t^{\ddes(\pi')}q^{2\dmaj(\pi')}\prod_{i\in\Nega(\pi')\setminus\{n\}}x_{2i-1}x_{2i} \\
=& \sum_{\pi'\in\D_{n}} \left(\sum_{A\subseteq[n]}(-1)^{|A|}t^{2|A|}q^{\sum_{i\in A}(4i-2)}\right) t^{\ddes(\pi')}q^{2\dmaj(\pi')}\prod_{i\in\Nega(\pi')\setminus\{n\}}x_{2i-1}x_{2i} \\
=& \prod_{i-1}^{n}(1-t^2q^{4i-2}) \sum_{\pi'\in\D_{n}}t^{\ddes(\pi')}q^{2\dmaj(\pi')}\prod_{i\in\Nega(\pi')\setminus\{n\}}x_{2i-1}x_{2i}
\end{align*}
and the proof is completed.
\qed

Next, we consider the signed Eulerian identity for $\D_{2n+1}$ with respect to $\ddes$.
For $\pi\in\D_{2n+1}$ let $i$ be the smallest integer such that the letters $2i-1$ and $2i$ satisfy one of the following conditions.
\begin{enumerate}
\item[(B1)] They are not in adjacent positions.
\item[(B2)] They have opposite signs, and are not both at the last two positions.
\item[(B3)] They are both at the last two positions with negative signs.
\item[(B4)] They are both at the last two positions with negative signs, and $\pi_{2n}<0$ and $\pi_{2n+1}<0$.
\end{enumerate}
Note that conditions (B3) and (B4) can be simply combined as ``are both at the last two positions and $\pi_{2n}<0$''; however, we separate it into two cases for the convenience of discussion.
Next, let $\iota(\pi)$ be the even-signed permutation obtained from $\pi$ by swapping the two letters $2i-1$ and $2i$.
For examples, $\iota(58\bar{7}\bar{1}\bar{2}963\bar{4})=68\bar{7}\bar{1}\bar{2}953\bar{4}$, $\iota(58\bar{7}\bar{1}\bar{2}96\bar{3}4)=58\bar{7}\bar{1}\bar{2}96\bar{4}3$, and $\iota(58\bar{7}\bar{1}\bar{2}96\bar{3}\bar{4})=58\bar{7}\bar{1}\bar{2}96\bar{4}\bar{3}$.
It is easy to see that $\iota$ is an involution on $\D_{2n+1}$ whose fixed points are those even-signed permutations $\pi$ having the following properties.
\begin{itemize}
\item For $i\in[n]$ the letters $2i-1$ and $2i$ are adjacent.
\item For $i\in[n]$ the letters $2i-1$ and $2i$ have the same sign if both of them are not at the last two positions.
\item If $2i-1$ and $2i$ appear at the last two positions for some $i$, then $\pi_{2n}>0$.
\end{itemize}
Let $\mathcal{F}_{2n+1}^D$ denote the set of fixed points under $\iota$.
For example,
\begin{align*}
\mathcal{F}_3^D=\{123,213,312,321,\bar{1}\bar{2}3,\bar{2}\bar{1}3,\bar{3}1\bar{2},\bar{3}2\bar{1}\}.
\end{align*}
Recall that $\sgm(\pi)$ records the sign of the maximum letter in $\pi$.
It is clear that, for $\pi\in\F_{2n+1}^D$, $\sgm(\pi)=1$ if and only if $\pi_{2n+1}<0$.

Let $\pi$ be a non-fixed point under $\iota$ and denote by $\pi'=\iota(\pi)$.
Assume $i$ is the smallest integer such that $2i-1$ and $2i$ satisfy (B1), (B2), (B3) or (B4).
Observe that $\iota$ is identical to $\eta$, which is defined on $\D_{2n}$, if the last condition (B4) is omitted.
So, we have $\ell_D(\pi)=\ell_D(\pi')\pm 1$ and $\ddes(\pi)=\ddes(\pi')$ if $2i-1$ and $2i$ satisfy (B1), (B2) or (B3).
If they satisfy (B4), it is easy to see that $\inv(\pi) = \inv(\pi')$ and $\sum_{i\in\Nega(\pi)}(\pi_i+1) = \sum_{i\in\Nega(\pi')}(\pi'_i+1)\pm 1$, which imply $\ell_D(\pi)=\ell_D(\pi')\pm 1$; and, $\ddes(\pi)=\ddes(\pi')$ due to $\Des_F(\pi)=\Des_F(\pi')$ and $\pi_{2n+1}>0$, $\pi'_{2n+1}>0$.
Therefore, we conclude that
\begin{equation}\label{eq:D-odd-ell-fix}
\sum_{\pi\in\D_{2n+1}} (-1)^{\ell_D} t^{\ddes(\pi)} = \sum_{\pi\in\mathcal{F}_{2n+1}^D} (-1)^{\ell_D} t^{\ddes(\pi)}.
\end{equation}

We are ready to derive the following result.

\begin{theorem}\label{thm:D-odd-length}
For any positive integer $n$, we have
\begin{align*}
\sum_{\pi\in\D_{2n+1}}(-1)^{\ell_D(\pi)}t^{\ddes(\pi)} = (1-t^2)^n\sum_{\pi\in\D_{n+1}}t^{\ddes(\pi)}.
\end{align*}
\end{theorem}
\proof
By \eqref{eq:D-odd-ell-fix} it suffices to consider the set $\mathcal{F}_{2n+1}^D$.
Introduce the bijective correspondence $\phi$ between elements in $\mathcal{F}_{2n+1}^D$ and elements in $\D_{n+1}$ with some letters marked a ``hat'', according to the following rules.
\begin{itemize}
\item Each pair of adjacent entries of type $\pm(2j-1),\pm 2j$ in $\mathcal{F}_{2n+1}^D$ but not at the last two positions is replaced by $\pm j$.
\item Each pair of adjacent entries of type $\pm 2j,\pm(2j-1)$ in $\mathcal{F}_{2n+1}^D$ but not at the last two positions is replaced by $\pm\hat{j}$.
\item The pair of entries of type $(2j-1),\pm 2j$ at the last two positions in $\mathcal{F}_{2n+1}^D$ is replaced by $j$.
\item The pair of entries of type $2j,\pm(2j-1)$ at the last two positions in $\mathcal{F}_{2n+1}^D$ is replaced by $\hat{j}$.
\item The entry $\pm(2n+1)$ in $\mathcal{F}_{2n+1}^D$ is replaced by $\pm(n+1)$.
\item After the above steps, if the number of negatives of the resulting permutation is odd, then change the sign of the last entry.
\end{itemize}
We denote by $\widehat{\D}_{n+1}$ the set of all hatted even-signed permutations on $[n+1]$.
For $\pi'\in\widehat{\D}_{n}$ let $\inv(\pi')$, $\Nega(\pi')$, $\nega(\pi')$, $\ddes(\pi')$ and $\dmaj(\pi')$ be defined by omitting the hats, and let $\L(\pi')$ and $\P(\pi')$ denote respectively the sets of the hatted letters and their positions.
Note that $|\L(\pi')|=|\P(\pi')|\leq n$ since the letter $n+1$ will not be hatted.
For examples, $\phi(21\bar{5}\bar{6}87\bar{9}4\bar{3})=\hat{1}\bar{3}\hat{4}\bar{5}\hat{2}$, $\phi(\bar{2}\bar{1}\bar{5}\bar{6}87\bar{9}3\bar{4})=\hat{\bar{1}}\bar{3}\hat{4}\bar{5}\bar{2}$ and $\phi(21\bar{5}\bar{6}87439)=\hat{1}\bar{3}\hat{4}\hat{2}\bar{5}$.

Pick $\pi\in\F_{2n+1}^D$ and define $\pi'=\phi(\pi)$.
We  first claim that
\begin{equation}\label{eq:D-odd-length-fold1}
\ell_D(\pi)\equiv |\P(\pi')| \qquad (\text{mod }2).
\end{equation}
We will only consider the case that $\sgm(\pi)=1$ (i.e., $\pi_{2n+1}<0$) since the other cases can be dealt with in the same way. In this case, we have $\pi_{2n+1}<0$ and  $\pi_{2n+1}\neq\overline{2n+1}$, which forces the pair $(\pi_{2n},\pi_{2n+1})$ to be $(2j,\overline{2j-1})$ or $(2j-1,\overline{2j})$ for some $1\leq j\leq n$ and $\overline{2n+1}$ must appear as a letter in $\pi$.
For convenience, let $\sigma$ (resp., $\sigma'$) be the resulting signed permutation (resp., hatted signed permutation) obtained from $\pi$ (resp., $\pi'$) by removing the letter $\overline{2n+1}$ (resp., $\overline{n+1}$).
Observe that there are $2(|\sigma'_n|-1)$ entries among $\{\sigma_1,\sigma_2,\ldots,\sigma_{2n-2}\}$ whose absolute values are smaller than $|\sigma_{2n}|$.
Recall that $\max(\pi)$ and $\max(\pi')$ denote the positions of $\overline{2n+1}$ and $\overline{n+1}$ in $\pi$ and $\pi'$, respectively. We need to distinguish two cases.

\smallskip

{\bf Case 1:} $\pi_{2n}<|\pi_{2n+1}|$.
If $\sigma'_n>0$, then $$\inv(\sigma)=4\inv(\sigma') + |\P^+(\sigma')| + \nega(\sigma') - |\P^-(\sigma')| + 2(|\sigma'_n|-1) +1;$$
otherwise, $\sigma'_n<0$ and $$\inv(\sigma)=4\inv(\sigma') + |\P^+(\sigma')| + \nega(\sigma') - |\P^-(\sigma')| - 2(|\sigma'_n|-1).$$
As $\max(\pi) = 2\max(\pi')-1$, we have
\begin{align*}
\inv(\pi) &= \inv(\sigma) + \max(\pi)-1 \\
&= 4\inv(\sigma') + |\P(\sigma')| - 2|\P^-(\sigma')| + \nega(\sigma') + 2\max(\pi') +
\begin{cases}
2|\sigma'_n|-3, & \text{if }\sigma'_n>0\\
-2|\sigma'_n|, & \text{if }\sigma'_n<0.
\end{cases}
\end{align*}
Moreover, since $\sigma_{2n}=-2|\sigma'_n|$, it follows that
\begin{align*}
\sum_{\pi_i<0}(\pi_i+1) &= -(2n+1) + \sum_{\sigma_i<0}(\sigma_i+1) \\
&= -2n + \sum_{\sigma'_i<0 \text{ and } i<n}(4\sigma'_i+3) +
\begin{cases}
-2\sigma'_n+1, & \text{if }\sigma'_n>0 \\
2\sigma'_n+1, & \text{if }\sigma'_n<0
\end{cases}\\
&= -2n + 2\sum_{\sigma'_i<0 \text{ and } i<n}(2\sigma'_i+1) + \nega(\sigma') +
\begin{cases}
-2\sigma'_n+1, & \text{if }\sigma'_n>0 \\
2\sigma'_n, & \text{if }\sigma'_n<0
\end{cases}.
\end{align*}
Hence we have
\begin{align*}
\ell_D(\pi) = \inv(\pi) - \sum_{\pi_i<0}(\pi_i+1) \equiv |\P(\sigma')| = |\P(\pi')| \quad (\text{mod }2).
\end{align*}


{\bf Case 2:} $\pi_{2n}>|\pi_{2n+1}|$.
By  similar arguments as in Case 1, we have
\begin{align*}
\inv(\pi) = 4\inv(\sigma') + |\P(\sigma')| - 2|\P^-(\sigma')| + \nega(\sigma') + 2\max(\pi') +
\begin{cases}
2|\sigma'_n|-4, & \text{if }\sigma'_n>0\\
-2|\sigma'_n|+1, & \text{if }\sigma'_n<0
\end{cases}
\end{align*}
and
\begin{align*}
\sum_{\pi_i<0}(\pi_i+1) &= -(2n+1) + \sum_{\sigma_i<0}(\sigma_i+1) \\
&= -2n + \sum_{\sigma'_i<0 \text{ and } i<n}(4\sigma'_i+3) +
\begin{cases}
-2\sigma'_n+2, & \text{if }\sigma'_n>0 \\
2\sigma'_n+2, & \text{if }\sigma'_n<0
\end{cases}\\
&= -2n + 2\sum_{\sigma'_i<0 \text{ and } i<n}(2\sigma'_i+1) + \nega(\sigma') +
\begin{cases}
-2\sigma'_n+2, & \text{if }\sigma'_n>0 \\
2\sigma'_n+1, & \text{if }\sigma'_n<0
\end{cases}.
\end{align*}

It also concludes that
\begin{align*}
\ell_D(\pi) = \inv(\pi) - \sum_{\pi_i<0}(\pi_i+1) \equiv |\P(\sigma')| = |\P(\pi')| \quad (\text{mod }2).
\end{align*}

Let $\rho=\pi_1\pi_2\cdots\pi_{2n}|\pi_{2n+1}|$ and $\rho'=\pi'_1\pi'_2\cdots\pi'_{n}|\pi'_{n+1}|$.
Similar to the discussions for getting the identity \eqref{eq:B-odd-absinv-fold2}, we have $\fdes(\rho) = \fdes(\rho')+2|\L(\rho')|$ and thus
\begin{equation}\label{eq:D-odd-length-fold2}
\ddes(\pi) = \fdes(\rho) = \fdes(\rho')+2|\L(\rho')| = \ddes(\pi')+2|\L(\pi')|.
\end{equation}
As $|\L(\pi')|=|\P(\pi')|$ for $\pi'\in\widehat{\D}_{n+1}$, combining \eqref{eq:D-odd-length-fold1} and \eqref{eq:D-odd-length-fold2} yields that
\begin{align*}
& \sum_{\pi\in\mathcal{F}_{2n+1}^D} (-1)^{\ell_D} t^{\ddes(\pi)} = \sum_{\pi'\in\widehat{\D}_{n+1}}(-1)^{|\L(\pi')|}t^{\ddes(\pi')+2|\L(\pi')|} \\
=& \sum_{\pi'\in\D_{n+1}} \left(\sum_{A\subseteq[n]}(-1)^{|A|}t^{2|A|}\right) t^{\ddes(\pi')} = (1-t^2)^n \sum_{\pi'\in\D_{n+1}}t^{\ddes(\pi')},
\end{align*}
as desired.
\qed

\section{More Sign-balance Distributions for $\B_n$ and $\D_n$}\label{sec:sign-balance}
In this section, we consider the sign-balance
\begin{align*}
\sum_{\pi\in\B_n\text{ or }\D_n}\chi(\pi)t^{\stat_1(\pi)}q^{\stat_2(\pi)},
\end{align*}
where $\chi$ is a one-dimensional character on $\B_n$ or $\D_n$, and the statistics $\stat_1$ and $\stat_2$ are type $B$ or $D$ descent number and major index, respectively.
We shall derive the generating functions of these sign-balances.

\subsection{Sign-balance results on $\B_n$}

Recall that, for $\pi\in\B_n$, $\des_B(\pi)$ and $\maj_B(\pi)$ are the type $B$ descent number and major index of $\pi$.

\begin{theorem}\label{thm:B-GF-length}
For any positive integer $n$, we have
\begin{align*}
\sum_{\pi\in\B_n}(-1)^{\ell_B(\pi)}t^{\des_B(\pi)}q^{\maj_B(\pi)} = \sum_{\pi\in\B_n}(-1)^{\nega(\pi)}t^{\des_B(\pi)}q^{\maj_B(\pi)} = \prod_{i=0}^{n-1}\left(1-tq^i\right).
\end{align*}
\end{theorem}
\proof
Let $k$ be the smallest integer such that $|\pi_k|\neq k$.
Assume $|\pi_i|=k$ for some $i$.
Note that $i>k$.
Then, define $\theta(\pi)$ to be the signed permutation obtained from $\pi$ by changing the sign of $\pi_i$.
For example, if $\pi=\bar{1}25\bar{4}\bar{3}$, then $\theta(\pi)=\bar{1}25\bar{4}3$, where $k=3$ and $i=5$; and, if $\pi=\bar{1}2\bar{3}45$, then $\theta(\pi)=\pi$ since no such $k$ exists.
It is clear that $\theta$ is an involution, where the fixed points are those signed permutations $\pi\in\B_n$ with $|\pi_1|=1,|\pi_2|=2,\ldots,|\pi_n|=n$.
Let $\mathcal{I}_n$ denote the set of fixed points.

Pick a non-fixed point $\pi\in\B_n\backslash\mathcal{I}_n$ under $\theta$.
Let $k$ be the smallest integer such that $|\pi_k|\neq k$ and denote by $\pi'=\theta(\pi)$.
Obviously, $\nega(\pi')=\nega(\pi)\pm 1$.
Meanwhile, it is not hard to see that
\begin{align*}
\ell_B(\pi') = \begin{cases}
\ell_B(\pi) + 2k-1, & \text{if } \pi_i=k, \\
\ell_B(\pi) - (2k-1), & \text{if } \pi_i=-k.
\end{cases}
\end{align*}
Furthermore, since $|\pi_{i-1}|>|\pi_i|=k$, $\theta$ preserves $\Des_B$, which yields
\begin{align*}
\des_B(\pi') = \des_B(\pi) \quad\text{and}\quad \maj_B(\pi')=\maj_B(\pi).
\end{align*}
It follows that
\begin{equation}\label{eq:B-neg-fix}
\sum_{\pi\in\B_n}(-1)^{\nega(\pi)}t^{\des_B(\pi)}q^{\maj_B(\pi)} = \sum_{\pi\in\mathcal{I}_n}(-1)^{\nega(\pi)}t^{\des_B(\pi)}q^{\maj_B(\pi)},
\end{equation}
and
\begin{equation}\label{eq:B-length-fix}
\sum_{\pi\in\B_n}(-1)^{\ell_B(\pi)}t^{\des_B(\pi)}q^{\maj_B(\pi)} = \sum_{\pi\in\mathcal{I}_n}(-1)^{\ell_B(\pi)}t^{\des_B(\pi)}q^{\maj_B(\pi)}.
\end{equation}

Since each fixed point $\pi$ in $\mathcal{I}_n$ is of the form $|\pi_i|=i$ for $i\in[n]$, we have $\Des_B(\pi)=\Nega(\pi)-1$, which implies that $\des_B(\pi)=\nega(\pi)$ and $\maj_B(\pi)=\sum_{i\in\Nega(\pi)}(i-1)$.
Hence, the right-hand-side of \eqref{eq:B-neg-fix} is equal to
\begin{align*}
\sum_{\pi\in\mathcal{I}_n}(-t)^{\nega(\pi)}q^{\sum_{i\in\Nega(\pi)}(i-1)} = \sum_{A\subseteq[n]}(-t)^{|A|}q^{\sum_{i\in A}(i-1)} =  \prod_{i=0}^{n-1}\left(1-tq^i\right).
\end{align*}
On the other hand, it is easy to see that $\inv(\pi)=\sum_{i\in\Nega(\pi)}(i-1)$ for $\pi\in\mathcal{I}_n$.
Then, $\ell_B(\pi)=\sum_{i\in\Nega(\pi)}(2i-1)$, and hence the right-hand-side of \eqref{eq:B-length-fix} is equal to
\begin{align*}
& \sum_{\pi\in\mathcal{I}_n}(-1)^{\sum_{i\in\Nega(\pi)}(2i-1)}t^{|\Nega(\pi)|}q^{\sum_{i\in\Nega(\pi)}(i-1)} \\
= & \sum_{A\subseteq[n]}(-1)^{\sum_{i\in A}(2i-1)}t^{|A|}q^{\sum_{i\in A}(i-1)} \\
= & \prod_{i=0}^{n-1}\left(1-tq^i\right).
\end{align*}
This completes the proof.
\qed

\begin{theorem}\label{thm:B-GF-absinv}
For any positive integer $n$, we have
\begin{align*}
\sum_{\pi\in\B_n}(-1)^{\inv(|\pi|)}t^{\des_B(\pi)}q^{\maj_B(\pi)} = \left(1+(-1)^{n-1}tq^{n-1}\right)\prod_{i=0}^{n-2}\left(1-tq^i\right).
\end{align*}
\end{theorem}
\proof
For $\pi\in\B_n$ let $i$ be the smallest integer such that the letters $2i-1$ and $2i$ satisfy one of the following conditions.
\begin{enumerate}
\item[(C1)] They are not in adjacent positions.
\item[(C2)] They are in adjacent positions and have opposite signs.
\item[(C3)] They are in adjacent positions and have the same sign, but are not both at the $(2i-1)$-th and $(2i)$-th positions.
\end{enumerate}
Then, let $\psi(\pi)$ be the signed permutation obtained from $\pi$ by swapping the two letters $2i-1$ and $2i$ if they satisfy (C1) or (C2), or swapping the two letters $2i-1$ and $2i$ and changing their signs if they satisfy (C3).
For examples, $\psi(\bar{2}13\bar{5}64)=\bar{1}23\bar{5}64$, $\psi(\bar{1}\bar{2}5\bar{3}6\bar{4})=\bar{1}\bar{2}5\bar{4}6\bar{3}$ and $\psi(\bar{2}\bar{1}6{3}{4}5)=\bar{2}\bar{1}6\bar{4}\bar{3}5$.
It is clear that $\psi$ is an involution on $\B_n$, where the fixed points are those signed permutations $\pi$ satisfying that, for $i=1,2,\ldots,\lfloor n/2\rfloor$, the letters $2i-1$ and $2i$ have the same sign and are both in the $(2i-1)$-th and $(2i)$-th positions.
Let $\mathcal{F}_n$ denote the set of fixed points.
For instances, $\mathcal{F}_2=\{12,\bar{1}\bar{2},21,\bar{2}\bar{1}\}$ and $\mathcal{F}_3=\{123,\bar{1}\bar{2}3,213,\bar{2}\bar{1}3,12\bar{3},\bar{1}\bar{2}\bar{3},21\bar{3},\bar{2}\bar{1}\bar{3}\}$.

Let $\pi$ be a non-fixed point under $\psi$.
Let $i$ be the smallest integer such that the letters $2i-1$ and $2i$ satisfy one of the conditions (C1) -- (C3) and let $\pi':=\psi(\pi)$.
Obviously, $\inv(|\pi'|)=\inv(\pi)\pm 1$.
Moreover, it is easy to see that $\Des_B(\pi')=\Des_B(\pi)$ if the letters $2i-1$ and $2i$ in $\pi$ satisfy either (C1) or (C2).
Suppose the letters $2i-1$ and $2i$ in $\pi$ satisfy (C3) and appear in $\pi_k$ and $\pi_{k+1}$.
Observe that $|\pi_{k-1}|>2i$, otherwise we could find a smaller integer $i'<i$ such that $2i'-1$ and $2i'$ satisfy one of (C1) -- (C3).
This implies $\Des_B(\pi')=\Des_B(\pi)$, which concludes that $\psi$ preserves $\des_B$ and $\maj_B$.
Therefore, we have
\begin{align*}
\sum_{\pi\in\B_n}(-1)^{\inv(|\pi|)}t^{\des_B(\pi)}q^{\maj_B(\pi)} = \sum_{\pi\in\mathcal{F}_n}(-1)^{\inv(|\pi|)}t^{\des_B(\pi)}q^{\maj_B(\pi)}.
\end{align*}

Now, we define a bijective correspondence between $\mathcal{F}_n$ and $\mathcal{I}_n$, which is given in the proof of Theorem~\ref{thm:B-GF-length} as the set of signed permutations $\pi\in\B_n$ with $|\pi_1|=1, |\pi_2|=2, \ldots, |\pi_n|=n$, according to the following rules.
\begin{itemize}
\item Each pair of adjacent entries of type $2i,2i-1$ in $\mathcal{F}_n$ is replaced by $2i-1,\overline{2i}$.
\item Each pair of adjacent entries of type $\overline{2i},\overline{2i-1}$ in $\mathcal{F}_n$ is replaced by $\overline{2i-1},2i$.
\end{itemize}
Again, denote by $\pi'$ the resulting signed permutation.
It is clear that $\Des_B(\pi)=\Des_B(\pi')=\Nega(\pi')-1$, and thus $\des_B(\pi)=\nega(\pi')$ and $\maj_B{\pi}=\sum_{i\in\Nega(\pi')}(i-1)$.
Observe that both the patterns $2i,2i-1$ and $\overline{2i},\overline{2i-1}$ in $\pi$ provide $1$ to $\inv(|\pi|)$, while the ones $2i-1,\overline{2i}$ and $\overline{2i-1},2i$ in $\pi'$ provide $1$ to $\nega(\pi)$.
Then,
\begin{align*}
\inv(|\pi|)=
\begin{cases}
\nega(\pi'), & \text{if $n$ is even or $n$ is odd and $\pi_n=n$}; \\
\nega(\pi')-1, & \text{if $n$ is odd and $\pi_n=\bar{n}$}.
\end{cases}
\end{align*}
Therefore, when $n$ is even, we have
\begin{align*}
\sum_{\pi\in\mathcal{F}_n}(-1)^{\inv(|\pi|)}t^{\des_B(\pi)}q^{\maj_B(\pi)} &= \sum_{\pi\in\mathcal{I}_n}(-1)^{|\Nega(\pi)|}t^{|\Nega(\pi)|}q^{\sum_{i\in\Nega(\pi)}(i-1)} \\
&= \sum_{A\subseteq[n]}(-1)^{|A|}t^{|A|}q^{\sum_{i\in A}(i-1)} = \prod_{i=0}^{n-1}\left(1-tq^i\right).
\end{align*}
Also, when $n$ is odd, the generating function turns out to be
\begin{align*}
& \sum_{\pi\in\mathcal{I}_n\atop\pi_n=n}(-1)^{|\Nega(\pi)|}t^{|\Nega(\pi)|}q^{\sum_{i\in\Nega(\pi)}(i-1)} + \sum_{\pi\in\mathcal{I}_n\atop\pi_n=\bar{n}}(-1)^{|\Nega(\pi)|-1}t^{|\Nega(\pi)|}q^{\sum_{i\in\Nega(\pi)}(i-1)} \\
= & \sum_{A\subseteq[n-1]}(-1)^{|A|}t^{|A|}q^{\sum_{i\in A}(i-1)} +  \sum_{A\subseteq[n-1]}(-1)^{|A|}t^{|A|+1}q^{n-1+\sum_{i\in A}(i-1)} \\
= & (1+tq^{n-1}) \prod_{i=0}^{n-2}\left(1-tq^i\right).
\end{align*}
\qed

\subsection{Sign-balance results on $D_n$}

For a signed permutation $\pi\in\D_n$, let $$\Des_D(\pi):=\{i:\,0\leq i\leq n-1,\pi_i>\pi_{i+1}\}$$
with respect to the natural linear order $$\bar{n}<\cdots<\bar{1}<1\cdots<n,$$ where $\pi_0:=-\pi_2$.
Following \cite{BB_05}, the \emph{type $D$ descent number} of $\pi$, denoted by $\des_D(\pi)$, is defined to be the cardinality of $\Des_D(\pi)$, and the \emph{type D major index} of $\pi$ is given by $$\maj_D(\pi):=\sum_{i\in\Des_D(\pi)}i.$$

The proof of the following result is similar to that of Theorem~\ref{thm:B-GF-absinv}.

\begin{theorem}\label{thm:D-GF}
For any positive integer $n\geq 2$ we have
\begin{align*}
\sum_{\pi\in\D_n}(-1)^{\ell_D(\pi)}t^{\des_D(\pi)}q^{\maj_D(\pi)} = \big(1+(-1)^{n-1}tq^{n-1}\big)\prod_{i=0}^{n-2}\left(1-tq^i\right),
\end{align*}
and $\sum_{\pi\in\D_1}(-1)^{\ell_D(\pi)}t^{\des_D(\pi)}q^{\maj_D(\pi)}=1$.
\end{theorem}
\proof
For $\pi\in\D_n$ let $i$ be the smallest integer such that the letters $2i-1$ and $2i$ satisfy one of the following conditions.
\begin{enumerate}
\item[(D1)] They are not in adjacent positions.
\item[(D2)] They are in adjacent positions and have opposite signs, but are not both at the first two positions if $i=1$.
\item[(D3)] They are in adjacent positions and have the same sign, but not both at the $(2i-1)$-th and $(2i)$-th positions.
\end{enumerate}
Then, let $\psi(\pi)$ be the signed permutation obtained from $\pi$ by swapping the two letters $2i-1$ and $2i$ if they satisfy (D1) or (D2), or swapping the two letters $2i-1$ and $2i$ and changing their signs if they satisfy (D3).
It is clear that $\psi$ is an involution on $\D_n$, where the fixed points are those even-signed permutations $\pi$ satisfying that the letters $1$ and $2$ are at the first two positions and, for $i=2,3,\ldots,\lfloor n/2\rfloor$, the letters $2i-1$ and $2i$ have the same sign and are both at the $(2i-1)$-th and $(2i)$-th positions.
Let $\F^D_n$ denote the set of fixed points.
Note that, for $\pi\in\F^D_n$, if $n$ is even, $\pi_1$ and $\pi_2$ have the same sign; while, if $n$ is odd, exactly one of $\pi$ and $\pi_2$ is negative if and only if $\pi_n$ is negative.
For instances, $\F^D_2=\D_2$ and $\mathcal{F}_3=\{123,1\bar{2}\bar{3},\bar{1}2\bar{3},\bar{1}\bar{2}3,213,2\bar{1}\bar{3},\bar{2}1\bar{3},\bar{2}\bar{1}3\}$.

Let $\pi$ be a non-fixed point under $\varphi$.
Let $i$ be the smallest integer such that the letters $2i-1$ and $2i$ satisfy one of conditions (D1) -- (D3), and denote by $\pi'=\psi(\pi)$.
If the letters $2i-1$ and $2i$ satisfy either (D1) or (D2), then $\ell_D(\pi')=\ell_D(\pi)\pm 1$.
Since $1$ and $2$ are not both at the first two positions, $-\pi'_2>\pi'_1$ if and only if $-\pi_2>\pi_1$.
This verifies that $\Des_D(\pi')=\Des_D(\pi)$.
Now, suppose the letters $2i-1$ and $2i$ satisfy (D3), and appear at $\pi_k$ and $\pi_{k+1}$.
Note that $k>1$.
In this case, $|\pi_h|>2i$ for $h>k+1$, so changing the signs of $\pi_k$ and $\pi_{k+1}$ will not affect the size relationship between them and all entries on the right.
Therefore, $\inv(\pi')=\inv(\pi)\pm(4i-4)$ and $\sum_{\pi'<0}(\pi'+1)=\sum_{\pi<0}(\pi+1)\mp(4i-3)$, and thus $\ell_D(\pi')=\ell_D(\pi)\pm(8i-7)$.
On the other hand, since $k>1$ and $|\pi_{k-1}|>2i$ (by the same argument in the proof of Theorem~\ref{thm:B-GF-absinv}), we have $\Des_D(\pi')=\Des_D(\pi)$.
Hence we conclude that
\begin{align*}
\sum_{\pi\in\D_n}(-1)^{\inv(|\pi|)}t^{\des_D(\pi)}q^{\maj_D(\pi)} = \sum_{\pi\in\F^D_n}(-1)^{\inv(|\pi|)}t^{\des_D(\pi)}q^{\maj_D(\pi)}.
\end{align*}

Recall that $\mathcal{I}_n$ is the set of signed permutations with $|\pi_i|=i$ for all $i$.
Define a bijective correspondence between $\F^D_n$ and $\mathcal{I}_n$ according to the following rules.
\begin{itemize}
\item Each pair of adjacent entries of type $2i,2i-1$ in $\F^D_n$ is replaced by $2i-1,\overline{2i}$.
\item Each pair of adjacent entries of type $\overline{2i},\overline{2i-1}$ in $\F^D_n$ is replaced by $\overline{2i-1},2i$.
\item Pair $1\bar{2}$ in $\F^D_n$ is replaced by $12$.
\item Pair $\bar{1}2$ in $\F^D_n$ is replaced by $\bar{1}\bar{2}$.
\item Pair $2\bar{1}$ in $\F^D_n$ is replaced by $1\bar{2}$.
\item Pair $\bar{2}1$ in $\F^D_n$ is replaced by $\bar{1}2$.
\end{itemize}
For $\pi\in\F^D_n$ let $\pi'$ denote the corresponding signed permutation in $\mathcal{I}_n$.
It is routine to verify that $\Des_D(\pi)=\Des_B(\pi')=\Nega(\pi')-1$, and thus $\des_D(\pi)=\nega(\pi')$ and $\maj_D(\pi)=\sum_{i\in\Nega(\pi')}\big(i-1\big)$.

Now, we consider the relationship between $\ell_D(\pi)$ and $\nega(\pi')$.
When $n$ is even or $n$ is odd and $\pi_n=n$, by the structure of $\pi\in\F^D_n$, $\inv(\pi)$ is equal to an even integer plus the number of adjacent pairs of type $2i,2i-1$ or $\overline{2i-1},\overline{2i}$, and $\sum_{i\in\Nega(\pi)}(\pi_i+1)$ is equal to an even integer plus the number of adjacent pairs of type $\overline{2i-1},\overline{2i}$ or $\overline{2i},\overline{2i-1}$.
This implies that
\begin{align*}
\ell_D(\pi) \equiv |\big\{\text{adjacent pairs of type }2i,2i-1 \text{ or }\overline{2i},\overline{2i-1}\big\}| \quad (\text{mod }2).
\end{align*}
Adjacent pairs of types $2i,2i-1$ and $\overline{2i},\overline{2i-1}$ in $\pi$ are respectively replaced by $2i-1,\overline{2i}$ and $\overline{2i-1},2i$, each of which supports $1$ to $\nega(\pi')$, so we have
\begin{align*}
\ell_D(\pi) \equiv \nega(\pi') \quad (\text{mod }2).
\end{align*}
When $n$ is odd and $\pi_n=\bar{n}$, either $\pi_1$ or $\pi_2$ is negative.
By a similar argument, we have
\begin{align*}
\ell_D(\pi) \equiv \ell_D(\pi_1\pi_2) + |\big\{\text{adjacent pairs of type }2i,2i-1 \text{ or }\overline{2i},\overline{2i-1}\big\}| \quad (\text{mod }2).
\end{align*}
Observe that $\ell_D(\pi_1\pi_2)\equiv\nega(\pi'_1\pi'_2)$ (mod 2), and the element $\bar{n}$ contributes an even number to $\ell_D(\pi)$ and $1$ to $\nega(\pi')$. 
It follows that
\begin{align*}
\ell_D(\pi) \equiv \nega(\pi')+1 \quad (\text{mod }2).
\end{align*}
This concludes that
\begin{align*}
(-1)^{\ell_D(\pi)}=
\begin{cases}
(-1)^{\nega(\pi')}, & \text{if $n$ is even or $n$ is odd and $\pi_n=n$}; \\
-(-1)^{\nega(\pi')}, & \text{if $n$ is odd and $\pi_n=\bar{n}$}.
\end{cases}
\end{align*}

Hence, the result follows by the same computations shown at the end of the proof of Theorem~\ref{thm:B-GF-absinv}.
\qed

\section{Concluding Remarks}\label{sec:conclusion}
In this paper we proposed some new signed Euler-Mahonian and signed Eulerian identities for $\B_n$, $\D_n$ and $G_{r,n}=G(r,1,n)$.
For $\B_n$ and $\D_n$, when $n$ is even we derived signed Euler-Mahonian identities on restricted (even-)signed permutations, generalizing Wachs and Biagioli's results \eqref{eq:signedEM-A} and \eqref{eq:signedM-B}.
When $n$ is odd we also obtained signed Eulerian identities on (even-)signed permutations, which generalize D\'{e}sarm\'{e}nien and Foata's result \eqref{eq:signedE-A-odd}.
Except for the odd $n$ cases, we further extend above results to $G_{r,n}$, where the `sign' is taken to be any one-dimensional character of the form $\chi_{1,b}$, for $0\leq b\leq r-1$.
The missing piece for odd $n$ seems elusive.
Moreover, in the last section we derived some neat closed forms for the sign-balance polynomials on $\mathcal{B}_n$ and $\mathcal{D}_n$, where the statistics are taken to be the Coxeter-type descent numbers and major indices.

Note that many identities of this work fit in the framework called the `folding phenomenon'~\cite{EFPT_15}, in which one has an identity of the form
$$\sum_{\pi \in \mathcal{X}_{2n} \text{ or } \mathcal{X}_{2n+1}} (-1)^{\mathsf{stat}_1(\pi)} q^{\mathsf{stat}_2(\pi)} =f(q) \sum_{\pi\in \mathcal{X}_n} q^{2\cdot \mathsf{stat}_2(\pi)},$$
where $\mathcal{X}_n$ is a family of combinatorial objects of size $n$ with statistics $\mathsf{stat}_1$ and $\mathsf{stat}_2$, and $f(q)$ is a rational function.
Note that in this viewpoint Theorem~\ref{thm:B-odd-length} is surprising since the statistic of the right hand side changes after the folding.
In the case of permutations there have been several instances of the folding phenomenon when $\mathcal{X}_n \subset \mathfrak{S_n}$ are certain families (say, involutions, 321-avoiding permutations, et al.), see~\cite{AR_04,EFPT_15,FHL_18} for examples.
It is then a natural question to investigate other types. We leave it to the interested readers.

\section*{Acknowledgments}

The authors would like to express their gratitude to the referees for their valuable comments and suggestions on improving the presentation of this paper.



\end{document}